\documentclass[a4paper,reqno, pdftex]{amsart}

\AtBeginDocument{%
  \def\MR#1{}
}

\usepackage{tikz, tikz-cd}
\usepackage{amsmath,amssymb}
\usepackage{amsthm}
\usepackage{mathrsfs}
\usepackage[shortlabels]{enumitem}
\usepackage{mathtools}
\usepackage[all,cmtip,2cell]{xy}
\usepackage{array}
\usepackage{comment}
\usepackage{braket}
\usepackage{stmaryrd}

\usepackage{url}
\usepackage{ascmac}
\usepackage{upgreek,euscript,dsfont}

\usepackage{graphicx}
\usepackage{multirow}

\usepackage[top=30truemm,bottom=30truemm,left=35truemm,right=35truemm]{geometry}

\usepackage[colorlinks=true,citecolor=green(pigment),linkcolor=alizarin,
urlcolor=denim]{hyperref}

\definecolor{alizarin}{rgb}{0.82, 0.1, 0.26}
\definecolor{azure(colorwheel)}{rgb}{0.0, 0.5, 1.0}
\definecolor{blue(pigment)}{rgb}{0.2, 0.2, 0.6}
\definecolor{denim}{rgb}{0.08, 0.38, 0.74}
\definecolor{mint}{rgb}{0.24, 0.71, 0.54}
\definecolor{parisgreen}{rgb}{0.31, 0.78, 0.47}
\definecolor{persiangreen}{rgb}{0.0, 0.65, 0.58}
\definecolor{seagreen}{rgb}{0.18, 0.55, 0.34}
\definecolor{shamrockgreen}{rgb}{0.0, 0.62, 0.38}
\definecolor{green(pigment)}{rgb}{0.0, 0.65, 0.31}

\theoremstyle{plain}
\newtheorem{thm}{Theorem}[section]
\newtheorem{prop}[thm]{Proposition}
\newtheorem{lem}[thm]{Lemma}

\newtheorem{cor}[thm]{Corollary}
\newtheorem*{thm*}{Theorem}

\theoremstyle{definition}
\newtheorem{dfn}[thm]{Definition}
\newtheorem{conj}[thm]{Conjecture}

\newtheorem{rem}[thm]{Remark}
\newtheorem{nota}[thm]{Notation}

\newtheorem{ex}[thm]{Example}

\theoremstyle{remark}

\numberwithin{equation}{section}

\newcommand{\ii}{\kern 1pt {\rm i}\kern 1pt }

\newcommand{\Z}{\mathbb{Z}}
\newcommand{\Q}{\mathbb{Q}}

\newcommand{\C}{\mathbb{C}}

\DeclareMathOperator{\rk}{rk}
\DeclareMathOperator{\id}{id}

\DeclareMathOperator{\Spec}{\mathrm{Spec}}

\DeclareMathOperator{\coh}{coh}

\DeclareMathOperator{\Mod}{Mod}

\DeclareMathOperator{\catmodif}{modif}
\DeclareMathOperator{\catMM}{MM}

\DeclareMathOperator{\catCT}{CT}

\DeclareMathOperator{\codim}{\mathrm{codim}}

\DeclareMathOperator{\add}{add}

\DeclareMathOperator{\refl}{ref}
\DeclareMathOperator{\Hom}{Hom}

\DeclareMathOperator{\End}{End}

\DeclareMathOperator{\Ker}{Ker}
\DeclareMathOperator{\Cok}{Cok}
\DeclareMathOperator{\Supp}{Supp}
\DeclareMathOperator{\Ext}{Ext}

\DeclareMathOperator{\Cl}{Cl}
\DeclareMathOperator{\Aut}{Aut}

\DeclareMathOperator{\Res}{Res}
\DeclareMathOperator{\Ind}{Ind}

\newcommand{\scrT}{\EuScript{T}}

\newcommand{\sfF}{{\sf F}}
\newcommand{\sfG}{{\sf G}}

\let\Im\relax
\DeclareMathOperator{\Im}{\mathrm{Im}}

\newcommand{\rp}{\,{\cdot}\,}

\DeclareMathOperator{\MM}{{\sf MM}}

\DeclareMathOperator{\NCCR}{\sf NCCR}
\DeclareMathOperator{\CT}{\sf CT}

\mathchardef\mhyphen="2D

\renewcommand{\l}{\langle}
\renewcommand{\r}{\rangle}

\newcommand{\simto}{\,\,\,\,\raisebox{4pt}{$\sim$}{\kern -1.1em\longrightarrow}\,}
\newcommand{\hookto}{\hookrightarrow}

\makeatletter
\newcommand*{\defeq}{\mathrel{\rlap{%
                    \raisebox{0.3ex}{$\m@th\cdot$}}%
                    \raisebox{-0.3ex}{$\m@th\cdot$}}%
                     =}
\makeatother

\makeatletter
\newcommand*{\defqe}{=\mathrel{\rlap{%
                    \raisebox{0.3ex}{$\m@th\cdot$}}%
                    \raisebox{-0.3ex}{$\m@th\cdot$}}%
                     }
\makeatother

\newcommand{\cF}{\mathcal{F}}

\newcommand{\cN}{\mathcal{N}}

\newcommand{\fmod}{\operatorname{mod}}

\DeclareMathOperator{\fl}{\mathrm{f\hspace{0.0001mm}l}}

\DeclareMathOperator{\bLam}{\bf \Lambda}

\newcommand{\modif}{\operatorname{\sf{modif}}}

\DeclareMathOperator{\proj}{proj}
\DeclareMathOperator{\Kb}{\mathrm{K}^{\mathrm{b}}}

\def\gldim{\mathop{\rm gl.dim}\nolimits}
\def\prdim{\mathop{\rm proj.dim}\nolimits}
\def\injdim{\mathop{\rm inj.dim}\nolimits}

\DeclareMathOperator{\CM}{CM}

\numberwithin{equation}{section}

\usepackage[mathlines]{lineno}
\usepackage{amsmath}

\usepackage{etoolbox}

\newcommand*\linenomathpatch[1]{%
  \cspreto{#1}{\linenomath}%
  \cspreto{#1*}{\linenomath}%
  \csappto{end#1}{\endlinenomath}%
  \csappto{end#1*}{\endlinenomath}%
}

\newcommand*\linenomathpatchAMS[1]{%
  \cspreto{#1}{\linenomathAMS}%
  \cspreto{#1*}{\linenomathAMS}%
  \csappto{end#1}{\endlinenomath}%
  \csappto{end#1*}{\endlinenomath}%
}

\expandafter\ifx\linenomath\linenomathWithnumbers
  \let\linenomathAMS\linenomathWithnumbers

  \patchcmd\linenomathAMS{\advance\postdisplaypenalty\linenopenalty}{}{}{}
\else
  \let\linenomathAMS\linenomathNonumbers
\fi

\linenomathpatch{equation}
\linenomathpatchAMS{gather}
\linenomathpatchAMS{multline}
\linenomathpatchAMS{align}
\linenomathpatchAMS{alignat}
\linenomathpatchAMS{flalign}

\newcommand{\m}{\mathfrak{m}}

\newcommand{\p}{\mathfrak{p}}
\newcommand{\q}{\mathfrak{q}}

\title[Equivariant NCCRs and $G$-MM modules]{Equivariant noncommutative crepant resolutions\\ and $G$-maximal modifying modules}

\author[Y.HIRANO]{Yuki Hirano}

\address[Y.Hirano]{Tokyo University of Agriculture and Technology, 2-24-16 Nakacho, Koganei, Tokyo 184-8588, Japan}
\email{hirano@go.tuat.ac.jp}

\begin{document}

\begin{abstract}
Let $S$ be a Cohen--Macaulay normal domain containing a field $k$ with an action of a finite group $G$ with $|G|$ invertible in $k$. We prove that  $G$-equivariant noncommutative crepant resolutions of $S$ induce noncommutative crepant resolutions of $S^G$ if the $G$-action is small. 
 We also discuss similar results for equivariant maximal modification algebras. We introduce the notion of $G$-maximal modifying ($G$-MM) module, and we show that  a $G$-MM module over $S$ induces an MM $S^G$-module in  some special cases. Furthermore, we prove that Gorenstein MMAs of three-dimensional $\Q$-Gorenstein Cohen--Macaulay rings are derived equivalent. As an application of these results, we prove the existence of an MM module and derived equivalences of MMAs for three-dimensional terminal singularities of index two.
\end{abstract}

\maketitle

\section{Introduction}

The notion of a noncommutative crepant resolution (NCCR) was introduced by Van den Bergh as a noncommutative ring analogue of a crepant resolution of  singularities \cite{VdB04a}. NCCRs have been studied for surface singularities \cite{Aus}, toric singularities \cite{VdB04a, boc, bro, hn, fms, svdb2, svdb3, nak, mat, tom, han, mal, ms1, ms2}, quotient singularities \cite{VdB04a, svdb1, ly, hh}, hypersurface singularities \cite{VdB02, Dao, IW14b}, cones over (Du Val) del Pezzo surfaces \cite{VdB04a, tom1, nvdb}, and symplectic singularities \cite{har}. Recently, NCCRs over quotient rings by regular sequences have been studied by Enomoto \cite{Eno}.

The notion of an NCCR was originally introduced for a Gorenstein ring $R$ as the endomorphism algebra 
\[
\bLam_M\defeq\End_R(M)
\] 
of some $M\in\refl R$ such that $\bLam_M\in\CM R$ and $\gldim\bLam_M<\infty$, where $\refl R$ denotes the category of  finitely generated reflexive $R$-modules.  Later, it was generalized for Cohen--Macaulay rings as follows:

\begin{dfn}[\cite{IW14a}] Let $R$ be a Cohen--Macaulay ring, and let $M\in\refl R$. 
\begin{itemize}
\item[(1)] $\bLam_M$ is an {\it NCCR} of $R$ if $\bLam_M\in\CM R$ and $\gldim(\bLam_M)_{\p}=R_{\p}$ for every $\p\in\Spec R$. If $\bLam_M$ is an NCCR, we say that $M$ {\it gives an NCCR}.
\item[(2)] $M$ is a {\it CT} module if $M\in\CM R$ and 
\[
\add M=\{N\in \CM R\,|\, \Hom(M,N)\in\CM R\}=\{L\in\CM R \,|\, \Hom(L,M)\in\CM R\}.
\] 
\end{itemize}
\end{dfn}
For a Cohen--Macaulay local ring $R$ with a canonical module,  $M\in\CM R$ is a CT module if and only if it gives an NCCR and $R\in\add M$. Dao--Iyama--Takahashi--Wemyss proved that the existence of an NCCR of $R$ requires $R$ to be $\Q$-Gorenstein \cite{DITW}. Furthermore, they proved that a $\Q$-Calabi--Yau ring $R$ of index $r$ containing a field of characteristic zero admits an NCCR if and only if its canonical cover $S$ admits an NCCR that arises from a $(\Z/r\Z)$-gradable $S$-module, where the grading of $S$ comes from the covering structure of $S$. If $k$ contains the $r$-th root of unity, the $(\Z/r\Z)$-grading of $S$ corresponds to a $(\Z/r\Z)$-action on $S$. A main ingredient  used in \cite{DITW} was a Gabriel cover, which cannot be applied for nonabelian group actions. This paper generalizes their setting to 
 arbitrary small finite group actions on Cohen--Macaulay normal domains containing a field.

 \vspace{2mm}
\subsection{Equivariant NCCRs}~
\vspace{2mm}

Let $S$ be a Cohen--Macaulay normal domain containing a field $k$, and let $G$ be a finite group such that  $|G|$ is invertible in $k$. Assume that $S$ has a canonical module. Consider a left $G$-action on $S$ given by a group homomorphism
\[
\upsigma\colon G\to \Aut_k(S); \quad g\mapsto \upsigma_g\colon S\simto S,
\]
where $\Aut_k(S)$ denotes the group of $k$-algebra automorphisms of $S$.  We set $R\defeq S^G$, and denote by
\[
\upvarphi\colon \Spec S\to \Spec R
\]
the affine quotient map. 
A {\it $G$-equivariant} $S$-module is defined to be a pair $(N,\theta)$ of an $S$-module $N$ together with a family $\theta=\{\theta_g\}_{g\in G}$ of isomorphisms $\theta_g\colon \upsigma_g^*N\simto N$ satisfying the cocycle condition. Such a family $\theta$ is called a {\it $G$-equivariant structure} on $N$. We say that $N\in\fmod S$ is {\it $G$-equivariantizable} if there is a $G$-equivariant structure $\theta$ on $N$. We denote by $\refl_GS$ the category of $G$-equivariant finitely generated reflexive $S$-modules. A {\it $G$-equivariant $S$-algebra} is an $S$-algebra $\Lambda$ with a $G$-action that is compatible with the $G$-action on $S$.  
For $(N,\theta)\in\refl_GS$, the endomorphism algebra $\bLam_N=\End_S(N)$ is a $G$-equivariant $S$-algebra, where the $G$-action is induced from $\theta$. A {\it $G$-equivariant NCCR} of $S$ is a $G$-equivariant $S$-algebra of the form $\bLam_N$ for some $(N,\theta)\in\refl_GS$ such that $\bLam_N$ is an NCCR of $S$.

First, we prove that a $G$-equivariant NCCR of $S$ induces an NCCR of $R$ when the $G$-action on $S$ is small. Here, 
we say that the $G$-action on $S$ is {\it small} if the free locus $U\subseteq \Spec S$ satisfies the condition ${\rm codim}(\Spec S\backslash U)\geq 2$. 

\begin{thm}\label{intro main 1}
 Let $(N,\theta)\in\refl_GS$. Assume that the $G$-action on $S$ is small.
\begin{itemize}
\item[$(1)$] $\bLam_N$ is an NCCR of $S$ if and only if $\bLam_{\upvarphi_*N}$ is an NCCR of $R$.
\item[$(2)$] $N$ is a CT $S$-module if and only if $\upvarphi_*N$ is a CT $R$-module. 
\end{itemize}
\end{thm}

If $S$ is regular, then the free module $S$ has a canonical $G$-equivariant structure and it is a CT $S$-module. Thus, Theorem \ref{intro main 1}\,(2) implies that $\upvarphi_*S$ is a CT $R$-module if $S$ is regular, which recovers  \cite[Theorem 5.7]{IW14a}. By Theorem \ref{intro main 1}\,(1), there is a map 
\[
\upvarphi_*\colon G\text{-}\NCCR S\to \NCCR R,
\]
where $\NCCR R$ (resp. $G\text{-}\NCCR S$) is the set of  additive equivalence classes of reflexive $R$-modules (resp. $G$-equivariantizable reflexive $S$-modules) giving NCCRs. Here, we say that two modules $N$ and $L$ are {\it additively equivalent} if $\add N=\add L$. Similarly, by Theorem \ref{intro main 1}\,(2), there is a map $\upvarphi_*\colon G\text{-}\CT S\to \CT R$, where $\CT R$ (resp. $G\text{-}\CT S$) is the set of  additive equivalence classes of CT $R$-modules (resp. $G$-equivariantizable CT $S$-modules).

\begin{thm}\label{intro main 2} Assume that the $G$-action on $S$ is small.
Then the functor $\upvarphi_*\colon \fmod S\to\fmod R$ induces injective maps
\[
\upvarphi_*\colon G\text{-}\NCCR S\hookto \NCCR R\quad\mbox{and}\quad\upvarphi_*\colon G\text{-}\CT S\hookto \CT R.
\]
\end{thm}

If $S$ is a canonical cover of $R$ and the $G$-action on $S$ is a cyclic group action induced from the covering, the injective map  $\upvarphi_*\colon G\text{-}\NCCR S\hookto \NCCR R$ is bijective by \cite{DITW}. However, the surjectivity does not hold in our generality. 

\vspace{2mm}
\subsection{$G$-maximal modifying modules}~
\vspace{2mm}

Although crepant resolutions and NCCRs do not exist in general,
 $\Q$-factorial terminalizations (also called {\it minimal models}) always exist for normal algebraic  varieties over $\C$ with only canonical singularities. Thus it is natural to consider not only an NCCR but also a noncommutative analogue of a minimal model. Iyama--Wemyss introduced the notion of a {\it maximal modification algebra} ({\it MMA}) as a noncommutative analogue of a minimal model \cite{IW14a}, and they proved that  MMAs  correspond to  minimal models at the level of derived categories of three-folds \cite[Theorem 4.16]{IW14b}. 
 
\begin{dfn}[\cite{IW14a}]
Let $S$ be a Cohen--Macaulay ring. A reflexive module $N\in\refl S$ is {\it modifying} if $\bLam_N\in\CM S$. Denote by 
$\catmodif S$ the category of modifying $S$-modules. A modifying module $N$ is {\it maximal modifying (MM)} if the following implication holds:
 \[
 N\oplus L\in\catmodif S \quad\Rightarrow \quad L\in\add N.
 \] 
 A {\it maximal modification algebra} ({\it MMA}) of $S$ is the endomorphism ring $\bLam_N$ of a maximal modifying $S$-module $N$.
\end{dfn}
 
If $N\in\catmodif S$ is not maximal, we can pick $L\in\catmodif S$ such that $L\not\in\add N$ and $N\oplus L\in \catmodif S$.  Since we do not know if this process terminates, the existence of an MM module (or equivalently, an MMA) is unclear. However, the following expectation seems natural,
when we expect that a noncommutative analogue of minimal model theory holds.

\begin{conj}[cf. {\cite[Conjecture 1.8\,(1)]{IW14b}}]\label{intro conj}
Let $\Spec R$ be a three-dimensional normal affine variety over $\C$ with only canonical singularities. Then the following hold:
\begin{itemize}
\item[(1)] $R$ admits an MMA. 
\item[(2)] $\widehat{R_{\m}}$ admits an MMA for every maximal ideal $\m$ of $R$.
\end{itemize}
\end{conj}

\begin{rem} Even if $M$ is an MM $R$-module,  the module $\widehat{M_{\m}}$ is not an MM $\widehat{R_{\m}}$-module in general \cite[Example 2.13]{IW14b}, and we do not know whether the implication  (1) $\Rightarrow$ (2) holds. Conversely, for a modifying $R$-module $M$, if  $\widehat{M_{\m}}$ is an MM $\widehat{R_{\m}}$-module for every maximal ideal $\m$ of $R$, then $M$ is an MM $R$-module. However, this does not necessarialy imply the direction (2) $\Rightarrow$ (1), since we do not know a priori that each MM module of $\widehat{R_{\m}}$ is given as the completion of  a common modifying $R$-module. 
\end{rem}

The following is a version of the noncommutative Bondal-Orlov conjecture.

\begin{conj}[cf. {\cite[Conjecture 1.8\,(2)]{IW14b}}, {\cite[Conjecture 1.4]{IW13}}]\label{intro conj 2}
Let $R$ be a three-dimensional Cohen--Macaulay equi-codimensional normal domain with a canonical module.   Then all MMAs are derived equivalent. 
\end{conj}

Conjecture \ref{intro conj} (1), (2) for Gorenstein terminal singularities were proved  in \cite{IW14b} using geometric results in \cite{VdB02}. A purely algebraic proof of Conjecture \ref{intro conj} (2) was given in \cite{IWtits}. However, there are few known examples of non-Gorenstein CM 3-folds that admit an MMA but no NCCR. Conjecture \ref{intro conj 2} was proved when $R$ admits an NCCR \cite{IW13} or $R$ is Gorenstein \cite[Theorem 4.16]{IW14a}.

We prove Conjecture \ref{intro conj} and Conjecture \ref{intro conj 2} for a certain family of non-Gorenstein singularities, which includes  terminal singularities of index two. A key ingredient for Conjecture  \ref{intro conj} is the notion of $G$-MMA, which is introduced below, and an MMA version of Theorem \ref{intro main 2}. 
In what follows, we use the same notation as in the previous section. 

\begin{dfn}
A $G$-equivariantizable modifying $S$-module $N$ is {\it $G$-maximal modifying} ({\it $G$-MM}) if  for every $G$-equivariantizable modifying $S$-module $L$ such that $N\oplus L\in\catmodif S$, we have $L\in\add N$. A {\it $G$-maximal modification algebra} ({\it $G$-MMA}) is the endomorphism ring $\bLam_N$ for some $G$-MM module $N$. 
\end{dfn}

It is obvious that a $G$-equivariantizable MM module  is a $G$-MM module. However, the converse is not true in general. Furthermore, it could occur that there are a $G$-MM module but no $G$-equivariantizable MM module. Denote by $\catmodif_GS$ the category of $G$-equivariant modifying $S$-modules, and write  $\MM R$ (resp. $G\text{-}\MM S$) for the set of additive equivalence classes of  MM $R$-modules (resp. $G$-MM $S$-modules). The following is an MM version of Theorem \ref{intro main 2}.

\begin{thm}\label{intro main 3}
Assume that the $G$-action on $S$ is small, and the functor $\upvarphi_*\colon \catmodif_GS\to \catmodif R$ is  dense. 
Then, the functor $\upvarphi_*\colon\fmod S\to\fmod R$ induces a bijective map
\[
\upvarphi_*\colon\mbox{$G$-}\MM S\simto \MM R.
\]
In particular, an MM $R$-module exists if and only if so does a $G$-MM $S$-module. 
\end{thm}

\vspace{2mm}
\subsection{Applications to three-dimensional terminal singularities of index two}~
\vspace{2mm}

Let $R$ be a Cohen--Macaulay normal domain with a canonical module. For a positive integer $r$, we say that $R$ is {\it $\Q$-Gorenstein} (resp. {\it $\Q$-Calabi--Yau}) {\it of index $r$} if 
\[
r=\min\left\{\,n>0\,\,\middle|\,\, \omega_R^{[n]}\text{ is projective (resp. free) for some canonical module $\omega_R$}\right\},
\]
where $\omega_R^{[n]}$ denotes the $n$-th reflexive power of $\omega_R$.
 If $R$ is  $\Q$-Calabi--Yau  of index $r$ and $\omega_R^{[r]}\cong R$ for a canonical module $\omega_R$, a {\it canonical cover} $S$ of $R$  is defined by 
\[
S\defeq R\oplus \omega_R\oplus \cdots\oplus \omega_R^{[r-1]},
\]
where the multiplication is induced from the reflexive products and an isomorphism $\omega_R^{[r]}\cong R$.

Assume that $\Spec R$ is a three-dimensional $\Q$-Calabi--Yau normal affine variety over $\C$ with only  terminal singularities. Then its canonical cover $S$ has only  Gorenstein terminal singularities. If $S$ is regular, then $R$ is a cyclic quotient singularity, and so the existence of an NCCR, and hence an MMA, of $R$ is known. On the other hand, if $S$ is singular, Conjecture \ref{intro conj} for $R$ is open. By Mori \cite{mor}, local models of terminal three-folds with singular canonical covers are classified into six types named
\[
cA/r,\quad cAx/2, \quad cAx/4, \quad cD/2, \quad cD/3, \quad cE/2.
\]
Here the number $r$ in each name $(-)/r$ stands for the index of $R$. As an application of Theorem \ref{intro main 3}, we prove Conjecture  \ref{intro conj}\,(2) for terminal singularities of index two. 

\begin{cor}\label{intro cor}
Let $\Spec R$ be a three-dimensional  normal affine variety over $\C$ with only terminal singularities. If the type of each singular point of $R$ is either $cA/2$, $cAx/2$, $cD/2$ or $cE/2$, then  $\widehat{R_{\m}}$ admits an MMA for every maximal ideal $\m$. 
\end{cor}

We also prove Conjecture \ref{intro conj 2} for arbitrary $\Q$-Gorenstein Cohen--Macaulay rings of index two. To this end, we generalize Iyama--Wemyss's derived equivalences of MMAs over three-dimensional Gorenstein normal domains to derived equivalences of Gorenstein MMAs over three-dimensional $\Q$-Gorenstein Cohen--Macaulay normal domains.

\begin{thm}\label{intro deq MMA}
Let $R$ be a three-dimensional Cohen--Macaulay equi-codimensional normal domain with a canonical module. If $R$ is $\Q$-Gorenstein,  then all Gorenstein MMAs are derived equivalent.
\end{thm}

If $R$ admits an NCCR, then it is $\Q$-Gorenstein, and every NCCR is a Gorenstein MMA. Thus Theorem \ref{intro deq MMA} is also a generalization of the proof of three-dimensional noncommutative Bondal--Orlov conjecture in \cite{IW13}.
As an application of Theorem \ref{intro deq MMA},  we prove Conjecture \ref{intro conj 2} for $\Q$-Gorenstein Cohen--Macaulay rings of index two. 

\begin{cor}\label{intro nc bo}
Let $R$ be a three-dimensional Cohen--Macaulay equi-codimensional normal domain with a canonical module. If $R$ is $\Q$-Gorenstein of index two,  then all MMAs are derived equivalent.
\end{cor}

Using Corollary \ref{intro nc bo}, we globalize Corollary \ref{intro cor} to prove Conjecture  \ref{intro conj}\,(1) for terminal singularities of index two.

\begin{cor}
Let $\Spec R$ be a three-dimensional normal affine variety over $\C$ with only terminal singularities. If the index of $R$ is two, then $R$ admits an MMA. 
\end{cor}

\vspace{2mm}
\subsection{Notation}~
\vspace{2mm}

Let $A$ and $B$ be a noetherian (not necessarily commutative) rings. We denote by $\Mod A$ (resp. $\fmod A$) the category of (resp. finitely generated) right $A$-modules. 
If $f\colon A\to B$ is a ring homomorphism, there are natural functors
\begin{align*}
f^{*}&\colon \Mod A\to \Mod B;\quad M\to M\otimes_AB\\
f_*&\colon \Mod B\to \Mod A;\quad N\to N_f,
\end{align*}
where $N_f$ is the abelian group $N$ together with the $A$-action given by $n\cdot a\defeq n\cdot f(a)$. Note that for another ring homomorphism $g\colon B\to C$, there are natural functor isomorphisms $(g\circ f)^*\cong g^*\circ f^*$ and $(g\circ f)_*\cong f_*\circ g_*$. 

\vspace{2mm}
\subsection{Acknowledgments}~
\vspace{2mm}

The author would like to thank Michael Wemyss and Osamu Iyama for giving valuable comments. The author was partially supported by JSPS Grant-in-Aid for Scientific Research (C) (26K06758). 

\vspace{2mm}
\subsection{Use of AI}~
\vspace{2mm}

The author used ChatGPT 6  only for searching references and finding typos and errors in the first version of this paper. 

\section{Equivariant modules over rings with finite group actions}

Throughout this section, $G$ is a finite group, and 
$\Lambda$ is a noetherian (not necessarily commutative) ring with a left $G$-action
\[
\upsigma\colon G\to \Aut(\Lambda);\quad g\mapsto \upsigma_g,
\] 
where $\Aut(\Lambda)$ is the group of ring automorphisms of $\Lambda$. We also assume that $|G|^{-1}\in\Lambda$. 

\vspace{2mm}
\subsection{Equivariant modules and modules with actions}~
\vspace{2mm}

We begin with the definition of equivariant modules. 
\begin{dfn}
A {\it $G$-equivariant} right $\Lambda$-module is a right $\Lambda$-module $N$ together with a family of  $\Lambda$-module isomorphisms 
\[
\theta=\left\{\theta_g\colon \upsigma_g^*N\simto N\right\}_{g\in G}
\] 
such that the cocycle condition $\theta_{gh}=\theta_g\circ\upsigma_g^*(\theta_h)$ holds for each $ g,h\in G$. In particular, $\theta_{1_G}=\id_N$ and $\theta^{-1}_g=\upsigma_{g}^*(\theta_{g^{-1}})$. We call the family $\theta$  a {\it $G$-equivariant structure} on $N$. 
\end{dfn}

\begin{nota}
For $g\in G$ and $N\in \fmod\Lambda$, we sometimes write $N\otimes_{\upsigma_g}\Lambda\defeq\upsigma_g^*N$, and write $x\otimes_{\upsigma_g}\lambda$, or more simply $x\otimes_g\lambda$,  for the element $x\otimes\lambda\in \upsigma^*_gN$. By definition, we have 
$x\lambda\otimes_g\lambda'=x\otimes_g\upsigma_g(\lambda)\lambda'$ for $x\in N$ and  $\lambda,\lambda'\in \Lambda$. 
\end{nota}

\begin{ex}
The ring $\Lambda$ has a canonical $G$-equivariant structure $\theta$ given by 
\[
\theta_g\colon \upsigma_g^*\Lambda\simto \Lambda;\quad \lambda\otimes1\mapsto \upsigma_g(\lambda).
\]
Indeed,  the equalities 
\[
\theta_g((\lambda\otimes_g1)\lambda')=\theta_g(\lambda\upsigma_g^{-1}(\lambda')\otimes_g1)=\upsigma_g(\lambda\upsigma_g^{-1}(\lambda'))=\theta_g(\lambda\otimes_g1)\lambda'
\]
show that $\theta_g$ is $\Lambda$-linear. 
Moreover, by  the equalities
\begin{align*}
(\theta_g\circ\upsigma_g^*(\theta_h))(\lambda\otimes_{{gh}} 1)&=\theta_g\bigl(\upsigma_g^*(\theta_h)((\lambda\otimes_{{h}} 1)\otimes_{g}1)\bigr)\\
&=\theta_g(\upsigma_h(\lambda)\otimes_{g}1)\\
&=\upsigma_{gh}(\lambda)=\theta_{gh}(\lambda\otimes_{{gh}} 1),
\end{align*}
 the cocycle condition  holds. 
\end{ex}

A {\it morphism} $\varphi\colon (N,\theta)\to (N',\theta')$ between $G$-equivariant right $\Lambda$-modules is a morphism $\varphi\colon N\to N'$ in $\fmod \Lambda$ such that $\varphi\circ \theta_g=\theta'_g\circ\upsigma^*_g(\varphi)$ for all $g\in G$. We denote by 
\[
\fmod_G\Lambda \quad({\rm resp.} \quad\proj_G\Lambda)
\]
the category of $G$-equivariant  finitely generated  (resp.  projective) right $\Lambda$-modules.

For $G$-equivariant $\Lambda$-modules $(N,\theta)$ and $(N',\theta')$, there is an  induced  left $G$-action on the abelian group $\Hom_{\Lambda}(N,N')$ given by 
\begin{equation}\label{left action}
g\cdot \varphi\defeq \theta'_g\circ \upsigma_g^*(\varphi)\circ \theta^{-1}_g.
\end{equation}
By definition, we have
\[
\Hom_{\fmod_G\Lambda}\bigl((N,\theta),(N',\theta')\bigr)=\Hom_{\Lambda}(N,N')^G.
\]

\begin{dfn}
 A {\it $G$-action} on $N\in\fmod\Lambda$ is a family of additive bijections
 \[
 \eta=\{\eta_g\colon N\simto N\}
 \]
 such that $\eta_{gh}=\eta_g\circ\eta_h$ and $\eta_g(x\lambda)=\eta_g(x)\upsigma_g(\lambda)$ for all $g,h\in G$, $x\in N$ and $\lambda\in\Lambda$.
\end{dfn}

A {\it morphism} $\varphi\colon (N,\eta)\to (N',\eta')$ between $\Lambda$-modules with $G$-actions is a morphism $\varphi\colon N\to N'$ in $\fmod\Lambda$ such that $\eta'_g\circ\varphi=\varphi\circ\eta_g$ for all $g\in G$. We denote by 
\[
\fmod(G,\Lambda)
\]
the category of finitely generated right $\Lambda$-modules with $G$-actions. 

\begin{ex}\label{ex theta eta}
(1) Let $(N,\theta)\in \fmod_G\Lambda$. For every $g\in G$, we define a bijection
\[
\eta^{\theta}_g\colon N\simto N;\quad x\mapsto \theta_{g}(x\otimes_g 1).
\]
Then  the cocycle condition of $\theta$ implies $\eta^{\theta}_{gh}=\eta^{\theta}_g\circ\eta^{\theta}_h$ for all $g,h\in G$. Moreover,
\[
\eta^{\theta}_g(x\lambda)=\theta_{g}(x\lambda\otimes_g1)=\theta_{g}(x\otimes_g\upsigma_{g}(\lambda))=\theta_{g}((x\otimes_g1)\upsigma_{g}(\lambda))=\eta^{\theta}_g(x)\upsigma_{g}(\lambda).
\]
 Hence, $\Phi(N,\theta)\defeq(N,\eta^{\theta})\in \fmod(G,\Lambda)$. \\
(2) Conversely, $(N,\eta)\in \fmod(G,\Lambda)$ induces a bijection 
\[
\theta^{\eta}_g\colon \upsigma^*_gN\simto N;\quad x\otimes_g1\mapsto \eta_{g}(x)
\]
for each $g\in G$.
Then we have 
\begin{align*}
\theta^{\eta}_g((x\otimes_g 1)\lambda)
&=\theta^{\eta}_g(x\otimes_g \lambda)
=\theta^{\eta}_g(x\upsigma^{-1}_g(\lambda)\otimes_g 1)
=\eta_{g}(x\upsigma^{-1}_g(\lambda))\\
&=\eta_{g}(x)\upsigma_{g}(\upsigma^{-1}_g(\lambda))=\theta_g^{\eta}(x\otimes_g1)\lambda.
\end{align*}
Thus $\theta^{\eta}_g$ is $\Lambda$-linear. Furthermore, the equalities
\begin{align*}
(\theta^{\eta}_g\circ\upsigma^*_g(\theta^{\eta}_h))(x\otimes_{{gh}}1)
=\theta^{\eta}_g(\eta_{h}(x)\otimes_{g}1)
=\eta_{g}(\eta_{h}(x))
=\theta^{\eta}_{gh}(x\otimes_{{gh}}1)
\end{align*}
show the cocycle condition of $\theta^{\eta}$. Therefore, $\Psi(N,\eta)\defeq(N,\theta^{\eta})\in\fmod_G\Lambda$. 
\end{ex}

 The constructions in Example \ref{ex theta eta} define functors 
 \[
 \Phi\colon\fmod_G\Lambda\to \fmod(G,\Lambda)\quad {\rm and}\quad \Psi\colon \fmod(G,\Lambda)\to \fmod_G\Lambda,
 \]
 and it is easy to check that $\Phi\circ \Psi\cong \id$ and $\Psi\circ\Phi\cong \id$. Thus we obtain the following.
 
\begin{prop}
The functors 
\begin{align*}
\Phi\colon&\fmod_G\Lambda\to \fmod(G,\Lambda); \quad (N,\theta)\mapsto (N,\eta^{\theta})\\
\Psi\colon&\fmod(G,\Lambda)\to \fmod_G\Lambda; \quad (N,\eta)\mapsto(N,\theta^{\eta})
\end{align*}
are exact equivalences and mutually inverse. 
\end{prop}

The following is an easy observation. 

\begin{lem}\label{sigma g h}
Let $N\in\fmod\Lambda$, and let $\varphi\colon N\to N$ be an additive map. Then, the map
\[
\upsigma_g^*N\to \upsigma_h^*N;\quad x\otimes_g 1\mapsto \varphi(x)\otimes_h1
\]
is $\Lambda$-linear if and only if $\varphi(x\lambda)=\varphi(x)\upsigma_{h^{-1}g}(\lambda)$ holds for each $x\in N$ and $\lambda\in \Lambda$.
\end{lem}

\vspace{2mm}
\subsection{Restriction, induction and invariant parts}~
\vspace{2mm}

There is a natural functor, called  the {\it restriction functor}, given by 
\[
\Res\colon \fmod_G\Lambda\to \fmod\Lambda;\quad (N,\theta)\mapsto N. 
\]
We also denote by $\Res\colon \fmod(G,\Lambda)\to \fmod\Lambda$ the corresponding restriction functor. A $\Lambda$-module is said to be {\it $G$-equivariantizable} if it lies in the essential image of the restriction functor $\Res\colon \fmod_G\Lambda\to \fmod\Lambda$.

For $N\in\fmod\Lambda$, we define 
\[
\Ind(N)\defeq\Bigl(\bigoplus_{g\in G}\upsigma^*_g(N), \theta\Bigr)\in \fmod_G\Lambda,
\]
where $\theta_g\colon \upsigma^*_g\bigl(\bigoplus_{h\in G}\upsigma^*_h(N)\bigr)\simto\bigoplus_{h\in G}\upsigma^*_h(N)$ is induced from natural isomorphisms 
\[
\upsigma^*_g(\upsigma^*_h(N))\simto \upsigma^*_{gh}(N).
\]
Then this defines a functor 
\[
\Ind\colon \fmod\Lambda\to\fmod_G\Lambda.
\]
 We also write $\Ind\colon \fmod\Lambda\to\fmod(G,\Lambda)$ for   $\Phi\circ\Ind$. The following is standard.

\begin{prop}[{\cite[Lemma 3.8]{Ela}}]
There are adjunctions
\[
\Res\dashv\Ind\dashv\Res.
\]
In particular, the functors $\Res$ and $\Ind$ are both exact. 
\end{prop}

Consider the $G$-invariant subring $\Lambda^G\defeq\{\lambda\in\Lambda\,|\, \upsigma_g(\lambda)=\lambda \mbox{ for all }g\in G\}\subseteq\Lambda$. Similarly, for $(N,\eta)\in \fmod(G,\Lambda)$, set
$
N^G\defeq\{x\in N\,|\, \eta_g(x)=x \mbox{ for all }g\in G\}.
$
Then  this defines a functor 
\[
(-)^G\colon \fmod(G,\Lambda)\to \fmod\Lambda^{G}.
\]
We also write $(-)^G\colon \fmod_G\Lambda\to\fmod\Lambda^G$ for the corresponding functor.

Assume that $|G|^{-1}\in\Lambda$. For $(N,\eta)\in\fmod(G,\Lambda)$, consider the Reynolds operator 
\[
\rho_N\colon N\to N^G;\quad x\mapsto \frac{1}{|G|}\sum_{g\in G}\eta_g(x).
\]
Denote by $N_G\defeq \Ker(\rho_N)\in\fmod\Lambda^G$.

\begin{prop}\label{inv exact}
Assume that $|G|^{-1}\in\Lambda$. 
\begin{itemize}
\item[$(1)$] For any $(N,\eta)\in\fmod(G,\Lambda)$, there is a decomposition $N=N^{G}\oplus N_G$ in $\fmod \Lambda^G$.
\item[$(2)$] For any $\varphi\colon (N,\eta)\to (N',\eta')$ in $\fmod(G,\Lambda)$, there is a decomposition 
\[
\varphi=\varphi^G\oplus \varphi_G,
\]
where $\varphi^G\defeq \varphi|_{N^G}\colon N^G\to N'^G$ and $\varphi_G\defeq \varphi|_{N_G}\colon N_G\to N'_G$.
\item[$(3)$] The functor $(-)^G\colon \fmod(G,\Lambda)\to \fmod\Lambda^{G}$ is exact. 
 \end{itemize}
\begin{proof}
(1) This follows from  $\rho_N\circ\iota_N=\id_{N^G}$, where $\iota_N\colon N^G\hookto N$ is the natural inclusion. \\
(2) It is enough to show that $\Im(\varphi^G)\subseteq N'^G$ and $\Im(\varphi_G)\subseteq N'_G$. If $x\in N^G$, then 
\[\eta'_g(\varphi(x))=\varphi(\eta_g(x))=\varphi(x)\]
for each $g\in G$. Thus $\varphi(x)\in N'^G$. 
Let $y\in N_G$. Then 
\[
\rho_{N'}(\varphi(y))=\frac{1}{|G|}\sum_{g\in G}\eta'_g(\varphi(y))=\frac{1}{|G|}\sum_{g\in G}\varphi(\eta_g(y))=\varphi(\rho_N(y))=0.
\]
Hence $\varphi(y)\in N'_G$. \\
(3) By (2), for any $\varphi\colon (N,\eta)\to (N',\eta')$, there are equalities
\[
\Ker(\varphi)=\Ker(\varphi^G)\oplus\Ker(\varphi_G) \quad \text{and}\quad \Im(\varphi)=\Im(\varphi^G)\oplus\Im(\varphi_G).
\]
In particular, the functor $(-)^G\colon \fmod(G,\Lambda)\to \fmod\Lambda^{G}$ is exact.
\end{proof}
\end{prop}

\begin{lem}\label{ind inv}
Let $N\in\fmod\Lambda$, and denote by $\upvarphi\colon \Lambda^G\hookto \Lambda$ a natural injection. Then there is a $\Lambda^G$-module isomorphism
\[
\upvarphi_*N\cong \Ind(N)^G.
\]
\begin{proof}
Write $\Ind(N)=(\bigoplus_{g\in G}\upsigma^*_gN,\eta)$. Then, for each $g\in G$, the bijection $\eta_g$ is given by 
\[
\eta_g\colon\bigoplus_{h\in G}\upsigma^*_hN\simto\bigoplus_{h\in G}\upsigma_h^*N;\quad\sum_{h\in G}(x_h\otimes_{h}1)\mapsto \sum_{h\in G}x_h\otimes_{{gh}}1.
\]
Thus $\eta_g(\sum_{h\in G}(x_h\otimes_{h}1))=\sum_{h\in G}(x_h\otimes_{h}1)$ if and only if $x_h=x_{gh}$ for all $h\in G$. Therefore, we have
\[
\Ind(N)^G=\left\{\sum_{g\in G}(x_g\otimes_{g}1)\,\middle|\, x_g=x_h \mbox{ for all }g,h\in G\right\},
\]
and so the diagonal map $\delta\colon N\to \Ind(N)^G;x\mapsto \sum_{g\in G}(x\otimes_g1)$ is an additive bijection.  Furthermore,  $\delta$ is $\Lambda^G$-linear. This finishes the proof.
\end{proof}
\end{lem}

\vspace{2mm}
\subsection{Skew group rings}~
\vspace{2mm}

Equivariant modules also correspond to modules over a skew group ring. The {\it skew group ring} 
\[
\Lambda\#G\defeq\bigoplus_{g\in G}\Lambda t_g
\]
 of $G$ over $\Lambda$ is a free $\Lambda$-module  with basis $\{t_g\}_{g\in G}$  and the multiplication is given by
\[
(\lambda t_g)(\lambda't_h)\defeq \lambda\upsigma_g(\lambda')t_{gh}.
\]
As usual, we write $1_{\Lambda}t_g=t_g$. Then, $t_gt_h=t_{gh}$ and $t_1=1_{\Lambda\#G}$. There is a natural ring homomorphism 
\[
\Lambda\to \Lambda\#G;\quad \lambda\mapsto \lambda t_1.
\]
\begin{ex}\label{skew action}
(1) Let $(N,\eta)\in\fmod(G,\Lambda)$. Then $N$ can be considered as a right $\Lambda\#G$-module by
\[
N\times \Lambda\#G\to N;\quad (x, \lambda t_g)\mapsto\eta_{g^{-1}}(x\lambda).
\]
We denote  this right $\Lambda\#G$-module by $(N,\eta)^{\#}$. \\
(2) Conversely, let $N\in \fmod\Lambda\#G$. Then $N$ can be considered as a right $\Lambda$-module by $x\lambda\defeq x(\lambda t_1)$. Furthermore, for each $g\in G$, there is an additive bijection
\[
\eta_g\colon N\to N; \quad x\mapsto xt_{g^{-1}}.
\]
Then $\eta_{gh}=\eta_g\circ\eta_h$,  and there are equalities
\[
\eta_g(x)\upsigma_g(\lambda)=(x t_{g^{-1}})(\upsigma_g(\lambda)t_1)=x\lambda t_{g^{-1}}=\eta_g(x\lambda).
\]
Thus $(N,\eta)\in \fmod(G,\Lambda)$. We put $^GN\defeq (N,\eta)$. 
\end{ex}

The constructions in Example \ref{skew action} induce functors
 \[
 (-)^{\#}\colon\fmod(G,\Lambda)\to \fmod\Lambda\#G\quad {\rm and}\quad ^G(-)\colon \fmod\Lambda\#G\to \fmod(G,\Lambda),
 \]
and it is easy to verify that $^G((N,\eta)^{\#})\cong (N,\eta)$ and $(^GN)^{\#}\cong N$. Therefore, the following holds.

\begin{prop}\label{action skew equiv}
The functors 
\begin{align*}
(-)^{\#}\colon&\fmod(G,\Lambda)\to \fmod\Lambda\#G\\
^G(-)\colon&\fmod\Lambda\#G\to \fmod(G,\Lambda)
\end{align*}
are exact equivalences and mutually inverse. 
\end{prop}

We also denote by 
\begin{align*}
(-)^{\#}\colon&\fmod_G\Lambda\simto \fmod\Lambda\#G\\
^G(-)\colon&\fmod\Lambda\#G\simto \fmod_G\Lambda
\end{align*}
the corresponding equivalences.

\begin{lem}\label{skew action free}
The following hold.
\begin{itemize} \item[$(1)$] There is an isomorphism $\Lambda\#G\cong \bigoplus_{g\in G}\upsigma_g^*\Lambda$ of  right  $\Lambda$-modules. In particular, $\Lambda\#G$ is a free right $\Lambda$-module of rank $|G|$. 
\item[$(2)$] If $|G|^{-1}\in\Lambda$,
then the right $\Lambda\#G$-module $\Lambda^{\#}$ is a direct summand of $\Lambda\#G$. 
\end{itemize}
\begin{proof}
(1) Note that there is a direct sum decomposition $\Lambda\#G\cong \bigoplus_{g\in G}\Lambda t_g$ as right $\Lambda$-modules. Consider the following additive bijection
\[
\varphi\colon \Lambda t_g\simto\upsigma_{g^{-1}}^*\Lambda;\quad \lambda t_g\mapsto \lambda\otimes_{g^{-1}}1.
\]
Then $\varphi((\lambda t_g)\lambda')=\varphi(\lambda\upsigma_g(\lambda')t_g)=\lambda\upsigma_g(\lambda')\otimes_{g^{-1}}1=\lambda\otimes_{g^{-1}}\lambda'=\varphi(\lambda t_g)\lambda'$, and so $\varphi$ is $\Lambda$-linear.\\ 
(2) We consider an additive map given by
\[
\iota\colon \Lambda^{\#}\to \Lambda\#G; \quad \lambda\mapsto \sum_{g\in G}\upsigma_g(\lambda)t_g.
\]
Then  equalities
\[
\iota\bigl(\lambda(\lambda't_g)\bigr)=\iota\bigl(\upsigma_{g^{-1}}(\lambda\lambda')\bigr)=\sum_{h\in G}\upsigma_{h(g^{-1})}(\lambda\lambda')t_h=\sum_{f\in G}\upsigma_{f}(\lambda)\upsigma_f(\lambda')t_{fg}=\iota(\lambda)(\lambda't_g)
\] 
show that $\iota$ is $\Lambda\#G$-linear. We can also check that the map
\[
\pi\colon \Lambda\#G\to \Lambda^{\#}; \quad \sum_{g\in G}\lambda_gt_g\mapsto \frac{1}{|G|}\sum_{g\in G}\upsigma_{g^{-1}}(\lambda_g)
\]
is $\Lambda\#G$-linear and $\pi\circ\iota=\id_{\Lambda^{\#}}$. This proves the result.
\end{proof}
\end{lem}

\begin{prop}\label{action ext}
Let $(N,\eta),(N',\eta')\in\fmod(G,\Lambda)$. Assume $|G|^{-1}\in\Lambda$. 
\begin{itemize}
\item[$(1)$] If $(N,\eta)$ is a projective object in $\fmod(G,\Lambda)$, then $N\in\proj\Lambda$.
\item[$(2)$] For each $i\in\Z$, there is an isomorphism
\[
\Ext^i_{(G,\Lambda)}((N,\eta),(N',\eta'))\cong \Ext^i_{\Lambda}(N,N')^G.
\]
\end{itemize}
\begin{proof}
(1) Assume that  $(N,\eta)$ is a projective object. Then by Proposition \ref{action skew equiv}, $(N,\eta)$ is a direct summand of $^G(\Lambda\#G^{\oplus n})$ for some $n>0$. By Lemma \ref{skew action free}\,(1), $N$ is a direct summand of a free $\Lambda$-module, and thus $N\in\proj\Lambda$.\\
(2) By Proposition \ref{action skew equiv}, there is a projective resolution
\[
\cdots\to (P_1,\eta_1)\xrightarrow{\delta_1}(P_0,\eta_0)\xrightarrow{\delta_0}(N,\eta)\to0
\]
of $(N,\eta)$ in $\fmod(G,\Lambda)$. By (1), 
\[
\cdots\to P_1\xrightarrow{\delta_1}P_0\xrightarrow{\delta_0}N\to0
\] 
is a projective resolution of $N$. Since the functor $(-)^G\colon \Mod(G,\Z\left[|G|^{-1}\right])\to \Mod\Z\left[|G|^{-1}\right]$ is exact by Proposition \ref{inv exact}\,(3), where the $G$-action on $\Z\left[|G|^{-1}\right]$ is trivial, we have
\begin{align*}
\Ext^i((N,\eta),(N',\eta'))&\cong H^i(\Hom((P_{\bullet},\delta_{\bullet}),(N',\eta')))\\
&\cong H^i(\Hom(P_{\bullet},N')^G)\\
&\cong H^i(\Hom(P_{\bullet},N'))^G\cong \Ext^i(N,N')^G.
\end{align*}
This finishes the proof. 
\end{proof}
\end{prop}

\begin{cor}\label{proj eq}
Let $(N,\eta)\in \fmod(G,\Lambda)$, and assume $|G|^{-1}\in\Lambda$.  Then the following are equivalent.
\begin{itemize}
\item[$(1)$] $N\in\proj\Lambda$. 
\item[$(2)$] $(N,\eta)^{\#}\in\proj\Lambda\#G$.
\item[$(3)$] $(N,\eta)$ is a projective object in $\fmod(G,\Lambda)$.
\end{itemize}
\begin{proof}
(2)$\Leftrightarrow$(3) follows from Proposition \ref{action skew equiv}, and  (3)$\Rightarrow$(1) is nothing but Proposition \ref{action ext}\,(1). (1)$\Rightarrow$(3) follows from Proposition \ref{action ext}\,(2).
\end{proof}
\end{cor}

\section{Equivariant algebras}

In this section,  $S$ is a commutative noetherian ring, and  
\[
\uprho\colon G\to \Aut(S); \quad g\mapsto \uprho_g
\]
is a group homomorphism defining a left $G$-action on $S$. We assume that $|G|^{-1}\in\Lambda$.

\vspace{2mm}
\subsection{Equivariant algebras}~
\vspace{2mm}

\begin{dfn}
A {\it $G$-action} on an $S$-algebra $\Lambda$ is a group homomorphism
\[
\upsigma\colon G\to \Aut(\Lambda); \quad g\mapsto \upsigma_g
\]
that is compatible with the $G$-action $\uprho$ on $S$. 
An $S$-algebra with a $G$-action is called a {\it $G$-equivariant} $S$-algebra.
\end{dfn}

\begin{rem}
If $\Lambda$ is a $G$-equivariant $S$-algebra, the category $\fmod_G\Lambda$ can be identified with the category $\coh [\Lambda/G]$ of finitely generated right $[\Lambda/G]$-modules, where $[\Lambda/G]$ is an algebra over the quotient stack $[\Spec S/G]$ (see \cite{Hir} for the details).
\end{rem}

In what follows, suppose that $\Lambda$ is a $G$-equivariant $S$-algebra. We can construct new $G$-equivariant $S$-algebras by considering endomorphism rings of $G$-equivariant modules. 

\begin{ex}
Let $(N,\theta)\in\fmod_G\Lambda$. Then the endomorphism ring
\[
\bLam_N\defeq\End_{\Lambda}(N)
\]
has a $G$-action given by \eqref{left action}. This action is compatible with the $G$-action $\uprho$ on $S$, and thus $\bLam_N$ is a $G$-equivariant $S$-algebra.  
\end{ex}

\begin{lem}\label{inv endo}
Let $(N,\theta)\in \fmod_G\Lambda$, and write $\Ind(N)=(\bigoplus_g\upsigma_g^*N, \eta)\in \fmod(G,\Lambda)$.   For each $g,h\in G$, let $\psi_{g,h}\in \Hom_{\Lambda}(\upsigma_g^*N,\upsigma_h^*N)$, and define an additive map $\eta_{g,h}\colon N\to N$ by the formula $\psi_{g,h}(x\otimes_{\upsigma_g}1)=\eta_{g,h}(x)\otimes_{\upsigma_{h}}1$. Then $\psi\defeq(\psi_{g,h})_{g,h\in G}\in \End_{\Lambda}(\bigoplus_g\upsigma_g^*N)^G$ if and only if $\eta_{g,h}=\eta_{fg,fh}$ for all $f,g,h\in G$. 
\begin{proof}
By definition, $\psi\in \End_{\Lambda}(\bigoplus_g\upsigma_g^*N)^G$ if and only if for all $h\in G$ the following diagram commutes:
 \[
\begin{tikzcd}
\bigoplus_g\upsigma_g^*N\ \arrow[rr, "\psi"]\arrow[d, "\eta^{-1}_h"']&&\bigoplus_g\upsigma_g^*N\\
\bigoplus_g\upsigma_g^*N\arrow[rr, "\psi"]&&\bigoplus_g\upsigma_g^*N\arrow[u, "\eta_h"]
\end{tikzcd} 
\]
Let $x\otimes_{g}1\in\upsigma_g^*N$. Then 
\begin{align*}
\left(\eta_h\circ\psi\circ(\eta_h)^{-1}\right)(x\otimes_g1)&=(\eta_h\circ\psi)(x\otimes_{h^{-1}g}1)\\
&=\eta_h\left(\bigoplus_{f\in G}\psi_{h^{-1}g,f}(x\otimes_{h^{-1}g}1)\right)\\
&=\bigoplus_{f\in G}\eta_{h^{-1}g,f}(x)\otimes_{hf}1.
\end{align*}
Since $\psi(x\otimes_g1)=\bigoplus_{e\in G}\eta_{g,e}(x)\otimes_e1$, we see that $\left(\eta_h\circ\psi\circ(\eta_h)^{-1}\right)(x\otimes_g1)=\psi(x\otimes_g1)$ if and only if $\eta_{h^{-1}g,f}(x)=\eta_{g,hf}(x)$ for all $h,f\in G$ and $x\in N$. This implies the result. 
\end{proof}
\end{lem}

\begin{prop}\label{skew end}
Let $(N,\theta)\in\fmod_G\Lambda$. Then there is a ring isomorphism
\[
\bLam_N\#G\simto\End_{\Lambda}\Bigl(\bigoplus_{g\in G}\upsigma_g^*N\Bigr)^G.
\]
\begin{proof}
Let $\varphi=\sum_{g\in G}\varphi_gt_g\in \bLam_N\#G$. Then we define an additive map
\[
\eta^{\varphi}_{g,h}\colon N\to N;\quad x\mapsto \varphi_{h^{-1}g}\left(\theta_{h^{-1}g}(x\otimes_{h^{-1}g}1)\right).
\]
Then $\eta^{\varphi}_{g,f}(x\lambda)=\eta^{\varphi}_{g,f}(x)\upsigma_{h^{-1}g}(\lambda)$ holds for each $\lambda\in\Lambda$. Thus, by Lemma \ref{sigma g h}, the map
\[
\psi^{\varphi}_{g,h}\colon \upsigma_g^*N\to \upsigma^*_hN;\quad x\otimes_g1\mapsto \eta_{g,h}^{\varphi}(x)\otimes_h1
\]
is $\Lambda$-linear. Furthermore, we have $\psi^{\varphi}_{g,h}=\psi^{\varphi}_{fg,fh}$ for all $f,g,h\in G$. Hence, by  Lemma \ref{inv endo}, there is  an additive map 
\[
\alpha\colon \Lambda_N\#G\to\End_{\Lambda}\Bigl(\bigoplus_{g\in G}\upsigma_g^*N\Bigr)^G
\] 
defined by 
$\alpha(\varphi)\defeq(\psi^{\varphi}_{g,h}\colon \upsigma_g^*N\to\upsigma^*_hN)$. Conversely, let $\psi=(\psi_{g,h})\in \End_{\Lambda}\Bigl(\bigoplus_{g\in G}\upsigma_g^*N\Bigr)^G$, and suppose that $\psi_{g,h}\colon \upsigma_g^*N\to\upsigma^*_hN$ is given by $\eta_{g,h}\colon N\to N$ in the sense that $\psi_{g,h}(x\otimes_g1)=\eta_{g,h}(x)\otimes_h1$. Define an additive map 
\[
\beta\colon \End_{\Lambda}\Bigl(\bigoplus_{g\in G}\upsigma_g^*N\Bigr)^G\to\Lambda_N\#G
\]
by $\beta(\psi)\defeq\sum_{g\in G}\varphi^{\psi}_gt_g$, where 
$\varphi^{\psi}_g(x)\defeq \eta_{hg,h}\left((\varepsilon^{-1}_{g}\circ\theta^{-1}_{g})(x)\right)$ with $h\in G$ arbitrary, where 
\[
\varepsilon_g\colon N\simto \upsigma_g^*N
\] 
is an additive bijection given by $\varepsilon_g(x)\defeq x\otimes_g1$. Since $\alpha\circ\beta=\id$ and $\beta\circ\alpha=\id$, we see that $\alpha$ is bijective.

 Now it is enough to show that $\alpha$ is a ring homomorphism. 
 Let $\varphi_g,\varphi'_h\in\bLam_N$, and consider $\varphi=\varphi_gt_g, \,\varphi'=\varphi'_ht_h\in \bLam_N\#G$. We only need to prove the equality $\alpha(\varphi)\circ\alpha(\varphi')=\alpha(\varphi\cdot\varphi')$. 
By definition, we have
\[
\eta^{\varphi}_{a,b}
=
\begin{cases}
\varphi_g(\theta_g(x\otimes_g1))&\mbox{ if }\,\, b^{-1}a=g\\
0&\mbox{ otherwise },
\end{cases}
\quad\mbox{and}\quad
\eta^{\varphi'}_{c,d}
=
\begin{cases}
\varphi'_h(\theta_h(x\otimes_h1))&\mbox{ if }\,\, d^{-1}c=h\\
0&\mbox{ otherwise }
\end{cases}
\]
Thus non-zero $\eta^{\varphi}_{a,b}\circ\eta^{\varphi'}_{c,d}$ is of the form
\[
\eta^{\varphi}_{f(h^{-1}),f(h^{-1}g^{-1})}\circ\eta^{\varphi'}_{f,f(h^{-1})}=\varphi_g\circ\theta_g\circ\varepsilon_g\circ\varphi_h\circ\theta_h\circ\varepsilon_h.
\]
On the other hand, we have 
 \begin{align*}
\varphi\cdot\varphi'&=\bigl(\varphi_g\circ (g\cdot\varphi_h)\bigr)t_{gh}=\Bigl(\varphi_g\circ\theta_g\circ\upsigma^*_g(\varphi_h)\circ\theta^{-1}_g\Bigr)t_{gh}.
 \end{align*}
Hence non-zero $\eta^{\varphi\cdot\varphi'}_{e,f}$ is of the form
\begin{align*}
\eta^{\varphi\cdot\varphi'}_{e,e(h^{-1}g^{-1})}&=\varphi_g\circ\theta_g\circ\upsigma^*_g(\varphi_h)\circ\theta^{-1}_g\circ\theta_{gh}\circ\varepsilon_{gh}\\
&=\varphi_g\circ\theta_g\circ\upsigma^*_g(\varphi_h)\circ\upsigma_g^*(\theta_{h})\circ\varepsilon_{gh},
\end{align*}
where the second equality follows from the cocycle condition of $\theta$. Since $\varepsilon_g\circ\varphi_h\circ\theta_h\circ\varepsilon_h=\upsigma^*_g(\varphi_h)\circ\upsigma_g^*(\theta_{h})\circ\varepsilon_{gh}$, we see that 
\[
\eta^{\varphi}_{f(h^{-1}),f(h^{-1}g^{-1})}\circ\eta^{\varphi'}_{f,f(h^{-1})}=\eta^{\varphi\cdot\varphi'}_{f,f(h^{-1}g^{-1})}.
\]
This proves the equality $\alpha(\varphi)\circ\alpha(\varphi')=\alpha(\varphi\cdot\varphi')$. 
\end{proof}
\end{prop}

\vspace{2mm}
\subsection{Equivariant reflexive equivalences}~
\vspace{2mm}

In this section, we assume that $S$ is a normal domain. 
Let $\Lambda$ be a  module finite $G$-equivariant $S$-algebra. We say that $N\in \fmod_G\Lambda$ is {\it reflexive} (resp. {\it maximal} {\it Cohen--Macaulay}) if it is reflexive (resp. maximal Cohen--Macaulay) as an $S$-module. 
We denote by 
\[
\refl_G\Lambda \quad({\rm resp.}\quad \CM_G\Lambda)
\]
the category of finitely generated $G$-equivariant reflexive (resp. maximal Cohen--Macaulay) right $\Lambda$-modules. If $G$ is trivial, we simply write $\refl \Lambda$ (resp. $\CM\Lambda$).

\begin{dfn}
We say that $N\in\refl\Lambda$ is a {\it height one progenerator} (resp. {\it height one projective}) if $N_{\p}$  is a progenerator (resp. projective) as a right $\Lambda_{\p}$-module for all $\p\in\Spec S$ with ${\rm ht}(\p)\leq1$.
\end{dfn}

Let $(N,\theta)\in\fmod_G\Lambda$ and $(N',\theta')\in\refl_G\Lambda$. Then $\Hom_{\Lambda}(N,N')$ has the left $G$-action given by \eqref{left action}, and it is reflexive by \cite[Proposition 2.4\,(1)]{IR08}. Hence $\Hom_{\Lambda}(N,N')\in\refl_G\bLam_N$. 
The following is a $G$-equivariant version of a reflexive equivalence.

\begin{lem}
Assume that  $\Lambda\in\refl\Lambda$, and let $(N,\theta)\in\refl_G\Lambda$. 
If $N$ is a height one progenerator, then 
\[
\Hom_{\Lambda}(N,-)\colon \refl_G\Lambda\to \refl_G\bLam_N
\]
is an equivalence. In particular, for any $(L,\theta)\neq0\in \refl_GS$, the functor 
\[\Hom_S(L,-)\colon  \refl_GS\to \refl_G\bLam_L\]
 is an equivalence. 
\begin{proof}
Set $\sfF\defeq\Hom_{\Lambda}(N,-)\colon \refl_G\Lambda\to \refl_G\bLam_N$, and consider 
\[
\sfG\defeq\Hom_{\bLam_N}(\sfF(\Lambda),-)\colon \refl_G\bLam_N\to \refl_G\Lambda.
\] 
By \cite[Proposition 2.4\,(2)]{IR08}, there are natural isomorphisms $\Res(\sfG(\sfF(X)))\simto \Res(X)$
 and $\Res(\sfF(\sfG(Y)))\simto \Res(Y)$ for $X\in \refl_G\Lambda$ and $Y\in\refl_G\bLam_N$.  Since these isomorphisms are compatible with $G$-equivariant structures, $\sfF$ and $\sfG$ are equivalences. The last statement follows from \cite[Lemma 2.5\,(3)]{IW14a}. 
\end{proof}
\end{lem}

In the rest of this section, we assume that $S$ contains a field $k$, and the $G$-action on $S$ is given by 
\[
\uprho\colon G\to \Aut_k(S),
\]
 where $\Aut_k(S)$ is the group of automorphisms of the commutative $k$-algebra $S$. Then $k\subseteq S^G$. We denote by $\upvarphi\colon S^G\hookto S$ the natural injection, and we also write
 \[
 \upvarphi\colon \Spec S\to \Spec S^G
 \]
for the induced quotient map by abuse of notation. 
\begin{dfn}
We say that the affine quotient $\upvarphi\colon \Spec S\to\Spec S^G$ is an {\it almost principal $G$-bundle} if there are open subsets $U\subseteq \Spec S$ and $V=\upvarphi(U)\subseteq \Spec S^G$ such that the complements $\Spec S\backslash U$ and $\Spec S^G\backslash V$ have codimension greater than or equal to two, $U$ is $G$-stable and the restriction $\upvarphi|_U\colon U\to V$ is a principal $G$-bundle (with respect to fpqc topology over $k$ \cite[Definition 2.6]{Has1}), where $G$ is considered as the constant group scheme over $k$.
\end{dfn}

  We say that the $G$-action on $S$ is {\it small} if the free locus $U$ satisfies the condition
\[
\codim(\Spec S\backslash U)\geq2,
\]
where $\codim(\emptyset)\defeq\infty$.

\begin{lem}\label{small principal}
 The $G$-action on $S$ is small if and only if  the affine quotient $\upvarphi\colon \Spec S\to \Spec S^G$ is an almost principal $G$-bundle.
\begin{proof}
By assumption, $G$ is LFF over $k$ in the sense of \cite[(11.11)]{Has}. Thus the result follows from \cite[Proposition 14.8]{Has}.
\end{proof}
\end{lem}

The following is another version of a reflexive equivalence. 
\begin{prop}[{\cite[Corollary 11.4]{Has}}]\label{inv refl}
If the $G$-action on $S$ is small,  the functor
\[
(-)^G\colon \refl_GS\to\refl S^G
\]
is an equivalence, and its inverse is given by $M\mapsto (M\otimes_{S^G}S)^{**}=(\upvarphi^*M)^{**}$.
\begin{proof}
This follows from Lemma \ref{small principal} and \cite[Corollary 11.4]{Has}.
\end{proof}
\end{prop}

By the previous reflexive equivalence, we obtain the following.
\begin{lem}\label{varphi ind}
Let $(N,\theta)\in\refl_GS$. If the $G$-action on $S$ is small,  there is an isomorphism
\[
(\upvarphi^*\upvarphi_*N)^{**}\cong \Ind(N)
\]
in $\refl_GS$. In particular, there is an isomorphism $(\upvarphi^*\upvarphi_*N)^{**}\cong \bigoplus_{g\in G}\upsigma_g^*N$ in $\refl S$, and thus we have $\add N=\add \bigl((\upvarphi^*\upvarphi_*N)^{**}\bigr)$. 
\begin{proof}
By Lemma \ref{ind inv}, there is an isomorphism $\upvarphi_*N\cong \Ind(N)^G$. Since $\Ind(N)\in\refl_GS$, there are isomorphisms $\Ind(N)\cong\left(\upvarphi^*\left(\Ind(N)^G\right)\right)^{**}\cong (\upvarphi^*\upvarphi_*N)^{**}$, where the first isomorphism follows from  Proposition \ref{inv refl}. The second assertion follows from the first assertion. The last assertion follows from the second, since $\theta_g\colon\upsigma^*_gN\simto N$ for each $g\in G$. 
\end{proof}
\end{lem}

\vspace{2mm}
\subsection{Auslander-type isomorphism}~
\vspace{2mm}

In this section, we assume that the ring $S$ is a normal domain containing a field $k$, and that the $G$-action on $S$ is given by $\uprho\colon G\to \Aut_k(S)$ so that  $k\subseteq R\defeq S^G$.  Write 
\[
\upvarphi\colon R\hookto S
\]
for the natural injection. By abuse of notation, we also denote by $\upvarphi\colon \Spec S\to \Spec R$  the affine quotient. 
The following is a version of Auslander's result \cite{Aus}, \cite[Theorem 5.12]{LW}. 

\begin{prop}\label{auslander}
Let $(N,\theta)\in \refl_GS$, and  assume that the $G$-action on $S$ is small.  Then there is an $R$-algebra isomorphism
\[
\End_{R}(\upvarphi_*N)\cong \End_S(N)\#G.
\]
\begin{proof} 
By Proposition \ref{inv refl}, Lemma \ref{varphi ind} and Proposition \ref{skew end}, there are isomorphisms
\[
\End_{R}(\upvarphi_*N)\cong\End_{\refl_GS}\bigl(\Ind(N)\bigr)\cong \End_S\Bigl(\bigoplus_{g\in G}\upsigma_g^*N\Bigr)^G\cong \End_S(N)\#G,
\]
where all isomorphisms are $R$-linear. 
\end{proof}
\end{prop}

\vspace{2mm}
\section{Equivariant NCCRs and $G$-MMAs}

In this section, $S$ is a Cohen--Macaulay normal domain containing a field $k$. Let $G$ be a finite group, and consider a left $G$-action $\upsigma\colon G\to \Aut_k(S)$ on $S$. 
We assume that $|G|$ is invertible in $k$. Set $R\defeq S^G$, and use the same notation
\[
\upvarphi\colon R\hookto S \quad \mbox{ and } \quad\upvarphi\colon \Spec S\to \Spec R
\]
for the natural injection and the  affine quotient. We also assume that the Cohen--Macaulay rings $S$ and $R$ have  canonical modules $\omega_S$ and $\omega_R$ respectively, and they are equi-codimensional. Here a noetherian commutative ring is said to be {\it equi-codimensional} if all maximal ideals  have the same height. 

\vspace{2mm}
\subsection{Equivariant modifying modules}~
\vspace{2mm}

Following \cite{IW14a}, we recall several definitions on modifying modules and their fundamental properties. We say that $N\in\refl S$ is  {\it modifying} if  $\bLam_N=\End_S(N)\in\CM S$. A {\it maximal modifying} ({\it MM}) $S$-module is a modifying $S$-module $N$ such that the assumption that  
$N\oplus L$  is modifying implies that $L\in\add N$. 
We write 
\[
\catmodif S\quad (\mbox{resp.} \quad \catMM S)
\]
for  the category of modifying (resp. MM) $S$-modules.

 A {\it modification algebra} of $S$ is an  $S$-algebra of the form $\bLam_N$ for some $N\in\catmodif S$, and  a {\it maximal modification algebra} ({\it MMA}) of $S$ is a modification algebra $\bLam_N$ with $N\in\catMM S$. A {\it  noncommutative crepant resolution (NCCR)} of $S$ is a modification algebra $\bLam_N$  that is {\it non-singular}, i.e., $\gldim\bLam_{N_{\p}}=\dim S_{\p}$ for all $\p\in\Spec S$. We say that $N$ {\it gives an NCCR (resp. MMA)} of $S$ if $\bLam_N$ is an NCCR (resp. MMA) of $S$. 
 
 We say that $N\in\CM S$ is a {\it CT module} if for every $L\in\CM S$ the following equivalences hold:
\[
L\in\add N\quad\Longleftrightarrow\quad \Hom(N,L)\in\CM S\quad\Longleftrightarrow\quad\Hom(L,N)\in \CM S
\]
We denote by $\catCT S$ the category of CT $S$-modules. The following is a useful characterization of a CT module.

\begin{prop}[{\cite[Theorem 5.4]{IW14a}}]\label{char ct}
Let $N\in\CM S$. Then $N$ is a CT module if and only if $S\in\add N$ and $\bLam_N$ is an NCCR of $S$. 
\end{prop}

 We say that $N,L\in\refl S$ are {\it additive equivalent}, denoted by $N\overset{\add}{\sim} L$, if $\add N=\add L$. The set 
 \[
 \modif S\quad (\mbox{resp.}\quad \MM S,\quad \NCCR S, \quad\CT S)
 \] 
 is defined to be the set of additive equivalence classes of  $S$-modules that are modifying (resp. maximal modifying, giving  NCCRs, CT). We define a partial order in the set $\modif S$ by 
 \[
 N<L \quad(\mbox{resp.} \quad N\leq L) \quad\Longleftrightarrow\quad \add N\subsetneq\add L\quad(\mbox{resp.} \quad\add N\subseteq\add L).
 \]
\begin{lem}\label{poset mm}
Let $N\in\catmodif S$. Then $N$ is a maximal modifying $S$-module if and only if $N$ is a maximal element in the partially ordered set $\modif S$. 
\begin{proof}
Assume that $N\in\catMM S$, and let $L\in \catmodif S$ with $N\leq L$. Then $\add N\subseteq\add L$. Thus $N\oplus L\in\add L$, and so $N\oplus L\in\catmodif S$. Since $N\in\catMM S$, we have $L\in\add N$. Hence, $\add L\subseteq \add N$, and thus $N$ is maximal. 

Conversely, if $N$ is maximal in $\modif S$ and $N\oplus L\in \catmodif S$, then $N\leq N\oplus L$. Since $N$ is maximal, $L\in\add (N\oplus L)=\add N$. Thus $N\in\catMM S$. 
\end{proof}
\end{lem}

For later use, we prepare the following lemma.

\begin{lem}[{\cite[Proposition 4.5, Proposition 2.17]{IW14a}}]\label{nccr mm}
Let $N\in\catmodif S$. 
\begin{itemize}
\item[$(1)$] If $N$ gives an NCCR of $S$, then $N\in\catMM S$. 
\item[$(2)$] The $S$-algebra $\bLam_N$ is non-singular (equivalently, $\bLam_N$ is an NCCR of $S$)  if and only if $\CM\bLam_N=\proj\bLam_N$.
\end{itemize}
\end{lem}

We introduce the following terminology. 

\begin{dfn}
We say that $S$ is {\it MM-finite} (resp. {\it MM-simple}) if   $\MM S$ is a non-empty  finite set (resp. has a unique element). 
\end{dfn}

\begin{ex} It is known that a 3-dimensional local hypersurface $\Q$-factorial isolated singularity $S$ is MM-simple, and MM $S$-modules are projective \cite{Dao}. Another example of a 3-dimensional MM-simple normal domain is a $cA_n$-singularity  $S'\defeq\C[\![x,y,z,w]\!]/\l xy-f(z,w)^2\r$, where $f(z,w)$ is irreducible.  In this example, the basic MM $S'$-module is $S'\oplus\l x,f\r$  \cite[Theorem 5.20]{IW18}. A non-Gorenstein example of an MM-simple ring is the quotient singularity $S''\defeq\C[\![x,y,z]\!]^G$, where $G\defeq\frac{1}{2}(1,1,1)$ (see \cite[Theorem 3.7]{DITW} or Corollary \ref{mm fact}).  Later, we will give an example of non-Gorenstein MM-simple ring which has no NCCR.
\end{ex}
 
Now we consider equivariant versions of the above notions.

\vspace{2mm}
\begin{dfn}~ 
\begin{itemize}
\item[(1)]  A $G$-equivariant $S$-module $(N,\theta)$ is called a {\it $G$-equivariant modifying (resp. maximal modifying, CT)} $S$-module if $N$ is a modifying (resp. maximal modifying, CT) $S$-module. We denote by 
\[
\catmodif_GS \quad (\mbox{resp.} \quad \catMM_G S, \quad\catCT_GS)
\]
the category of $G$-equivariant modifying (resp. MM, CT) $S$-modules.

\item[(2)] We say that a $G$-equivariant $S$-algebra $\Lambda$ is a {\it $G$-equivariant  modification algebra} (resp. {\it $G$-equivariant MMA}) of $S$ if there exists $(N,\theta)\in \catmodif_GS$ (resp. $(N,\theta)\in \catMM_G S$) such that $\Lambda\cong \bLam_N$ as $G$-equivariant $S$-algebras. A {\it $G$-equivariant NCCR} of $S$ is similarly defined. 
\end{itemize}
\end{dfn}

\vspace{2mm}
\subsection{NCCRs of invariant rings from equivariant NCCRs}~
\vspace{2mm}

We prove that the existence of an equivariant NCCR implies the existence of an NCCR of $R$ if the $G$-action is small. First, recall the following.

\begin{lem}[{\cite[Proposition 1.8]{Yos}}]\label{CM iff CM}
Let $N\in\fmod S$. Then $N\in \CM S$ if and only if $\upvarphi_*N\in \CM R$. 
\end{lem}

The following is the first main result that generalizes \cite[Theorem 2.11]{DITW}.

\begin{thm}\label{main 1}
Let $(N,\theta)\in\refl_GS$, and assume that the $G$-action on $S$ is small.
\begin{itemize}
\item[$(1)$] $(N,\theta)\in\catmodif_GS$ if and only if $\upvarphi_*N\in\catmodif R$. 
\item[$(2)$] $\bLam_{N}$ is a $G$-equivariant NCCR of $S$ if and only if $\bLam_{\upvarphi_*N}$ is an NCCR of $R$. 
\item[$(3)$] $(N,\theta)\in \catCT_GS$ if and only if $\upvarphi_*N\in \catCT R$.
\end{itemize}
\begin{proof}
(1) By Proposition \ref{auslander}, there is an $R$-algebra isomorphism $\bLam_{\upvarphi_*N}\cong\bLam_N\#G$. Since $\bLam_N\#G$ is a free $\bLam_N$-module by Lemma \ref{skew action free}, $\bLam_{\upvarphi_*N}\in\CM R$ if and only if $\upvarphi_*\bLam_N\in\CM R$. By Lemma \ref{CM iff CM}, $\upvarphi_*\bLam_N\in\CM R$  if and only if $N\in\catmodif S$. \\
(2) By (1) and Lemma \ref{nccr mm}\,(2), it is enough to show the equivalence
\[
\CM\bLam_N=\proj\bLam_N\quad\Longleftrightarrow\quad\CM\bLam_{\upvarphi_*N}=\proj\bLam_{\upvarphi_*N}
\]
under the assumption that $\bLam_N$ and $\bLam_{\upvarphi_*N}$ are modification algebras. By Proposition \ref{auslander}, there is an $R$-algebra isomorphism 
\[
\gamma\colon \bLam_{\upvarphi_*N}\simto\bLam_N\#G.
\]
($\Rightarrow$) Assume that $\CM\bLam_N=\proj\bLam_N$. We have $\proj\bLam_{\upvarphi_*N}\subseteq\CM\bLam_{\upvarphi_*N}$, since $\bLam_{\upvarphi_*N}\in\CM\bLam_{\upvarphi_*N}$.   For the other inclusion, let $M\in\CM\bLam_{\upvarphi_*N}$. Then $\gamma^*M\in \CM (\bLam_N\#G)$. Consider the natural $R$-algebra homomorphism
\[
\iota\colon \bLam_N\to\bLam_N\#G;\quad \lambda\mapsto \lambda t_1.
\]
Then $\iota_*\gamma^*M\in \CM S$ since $\upvarphi_*\iota_*\gamma^*M\cong \gamma^*M\in\CM R$. This means that $\iota_*\gamma^*M\in\CM\bLam_N$. By  assumption, we have $\iota_*\gamma^*M\in\proj\bLam_N$.  Take $(L,\eta)\in\fmod(G,\bLam_N)$ with $(L,\eta)^{\#}\cong \gamma^*M$. Then $L\cong \Res(^G\gamma^*M)\cong \iota_*\gamma^*M\in\proj\bLam_N$. By Corollary \ref{proj eq}, we have $\gamma^*M\cong (L,\eta)^{\#}\in\proj\bLam_N\#G$. Hence $M\in\proj\bLam_{\upvarphi_*N}$. \\
($\Leftarrow$) Assume $\CM\bLam_{\upvarphi_*N}=\proj\bLam_{\upvarphi_*N}$. It is enough to show that $\CM\bLam_N\subseteq\proj\bLam_N$. Let $L\in \CM\bLam_N$. Then $L$ is a direct summand of $\Res\Ind(L)\in \CM\bLam_N$. Thus we only need to show that $\Res\Ind(L)\in\proj\bLam_N$. Since $\Res\Ind(L)\cong\Ind(L)^{\#}$ as $R$-modules, we see that $\Ind(L)^{\#}\in\CM(\bLam_N\#G)$. Hence $\gamma_*(\Ind(L)^{\#})\in \CM\bLam_{\upvarphi_*N}$. By assumption, $\gamma_*(\Ind(L)^{\#})\in\proj\bLam_{\upvarphi_*N}$, and so $\Ind(L)^{\#}\in\proj(\bLam_N\#G)$. Again by Corollary \ref{proj eq}, $\Res\Ind(L)\in\proj\bLam_N$. This completes the proof.\\
(3) Assume that $(N,\theta)\in\catCT_GS$. Since $|G|$ is invertible in $S$, the natural inclusion $R\hookto S$ splits via the Reynolds operator 
\[
S\to R;\quad s\mapsto \frac{1}{|G|}\sum_{g\in G}\upsigma_g(s).
\]
In particular,  $R$ is a direct summand of the $R$-module $S$. 
By (2) and Proposition \ref{char ct}, we have $\upvarphi_*N\in\catCT R$.

Conversely, assume that $\upvarphi_*N\in\catCT R$. By Lemma \ref{CM iff CM}, $N\in\CM S$. Since $\upvarphi_*N$ is modifying, so is $N$ by (1). Thus  the implications
\begin{align*}
L\in\add N&\Rightarrow \Hom(N,L)\in\CM S\\
L\in\add N&\Rightarrow \Hom(L,N)\in\CM S
\end{align*}
are automatic. It suffices to show the converses. First, note that we may assume that $L$ is equivariantizable. Indeed, since there are $S$-module  isomorphisms $\theta_g\colon \upsigma^*_gN\simto N$ for all $g\in G$, we have 
$
\add \Bigl(\bigoplus_{g\in G}\upsigma_g^*N\Bigr)=\add N,
$
 and $(\upsigma_g)_*N\cong(\upsigma_g)_*\upsigma_g^*N\cong N$.  Thus
 \[
 L\in\add N\Longleftrightarrow \bigoplus_{g\in G}\upsigma_g^*L\in\add N,
 \]
 and the equivalences
\begin{align*}
\Hom(\upsigma_g^*L,N)\in\CM S\quad&\Longleftrightarrow\quad \Hom(L,N)\in\CM S \\ \Hom(N,\upsigma_g^*L)\in\CM S\quad&\Longleftrightarrow\quad \Hom(N,L)\in \CM S
\end{align*}
hold for all $g\in G$. 
Therefore, we may replace $L$ by $\bigoplus_{g\in G}\upsigma_g^*L\cong \Res\Ind(L)$, which is $G$-equivariantizable. 

Now assume $\Hom(N,L)\in\CM S$ for some $(L,\theta)\in\refl_GS$. Then there are $R$-module isomorphisms
\begin{align*}
\Hom_R(\upvarphi_*N,\upvarphi_*L)&\cong\upvarphi_*\Hom_S(\upvarphi^*\upvarphi_*N,L)\\
&\cong \upvarphi_*\Hom_S((\upvarphi^*\upvarphi_*N)^{**},L)\\
&\cong \upvarphi_*\Hom_S\Bigl(\bigoplus_{g\in G}\upsigma^*_gN,L\Bigr)\\
&\cong\upvarphi_*\Hom_S(N,L)^{\oplus |G|},
\end{align*}
where the second isomorphism follows since $L$ is reflexive, the third one follows from Lemma \ref{varphi ind}, and the last one follows since there are isomorphisms $\theta_g\colon \upsigma^*_gN\simto N$ for all $g\in G$. Therefore, by Lemma \ref{CM iff CM}, we have $\Hom_R(\upvarphi_*N,\upvarphi_*L)\in\CM R$. By assumption, $\upvarphi_*L\in\add \upvarphi_*N$. Then $(\upvarphi^*\upvarphi_*L)^{**}\in \add (\upvarphi^*\upvarphi_*N)^{**}$, which is equivalent to 
\[
\bigoplus_{g\in G}\upsigma_g^*L\in \add \Bigl(\bigoplus_{g\in G}\upsigma_g^*N\Bigr)=\add N.
\]
Hence $L\in \add N$. This proves the implication $\Hom(N,L)\in\CM S\Rightarrow L\in\add N$. The identical argument works for the other implication, since $L$ is also equivariantizable. 
\end{proof}
\end{thm}

\begin{rem}
In Theorem \ref{main 1}, the assumption of smallness is necessary. For example, let $S\defeq\C[x,y,z,w]/\l xy-zw\r$, and consider $G=\frac{1}{2}(1,0,0,1)$.  Then, the invariant ring
\[
R\cong \C[X,Y,Z,W,V]/\l XW-V^2,XY-ZV,YV-ZW\r
\] 
is a toric ring, and it is easy to check that $R$ is not $\Q$-Gorenstein. 
Since $\Spec S$ contains non-free locus $\{x=w=0\}\cong\C^2$, the $G$-action on $S$ is not small. 
Although $M\defeq S\oplus\l x,z\r\in\refl S$, which is $G$-equivariantizable, gives an NCCR of $S$, the ring $R$ has no NCCR by \cite[Theorem 1.2]{DITW}.
\end{rem}

The following is a version of \cite[Theorem 5.7]{IW14a}, where it is assumed that the characteristic of the base field is zero and $S$ is a polynomial ring over a field. The following version drops the assumption of the characteristic of a base field, but assumes the smallness of $G$. 
\begin{cor}
If the $G$-action is small and $S$ is regular, then $\upvarphi_*S$ is a CT $R$-module.
\begin{proof}
By Proposition \ref{char ct}, we have  $S\in \catCT S$. Since $S$ has a canonical $G$-equivariant structure, $\upvarphi_*S\in\catCT R$ by  Theorem \ref{main 1}\,(3).
\end{proof}
\end{cor}

 Next, consider the set of additive equivalence classes
\[
\mbox{$G$-}\modif S\defeq\{N\in \catmodif S\,|\, \mbox{$N$ is $G$-equivariantizable}\}/\overset{\add}{\sim}
\]
of $G$-equivariantizable modifying $S$-modules. Similarly, we  denote by 
\[
\mbox{$G$-}\NCCR S\quad (\mbox{resp. }  \quad\mbox{$G$-}\CT S)
\] 
the set of additive equivalence classes of $G$-equivariantizable $S$-modules giving an NCCR (resp.   $G$-equivariantizable CT $S$-modules).  

\begin{cor}\label{modif inj}
Assume that the $G$-action on $S$ is small. 
The functor $\upvarphi_*\colon\fmod S\to\fmod R$ induces injective maps
\[
\mbox{$G$-}\modif S\hookto \modif R, \quad \mbox{$G$-}\NCCR S\hookto \NCCR R\quad\mbox{and}\quad\mbox{$G$-}\CT S\hookto \CT R.
\]
\begin{proof}
 By Theorem \ref{main 1}, it is enough to show that the map $\upvarphi_*\colon\mbox{$G$-}\modif S\to \modif R$ is injective. For $N,N'\in \mbox{$G$-}\modif S$, assume that $\add(\upvarphi_*N)=\add(\upvarphi_*N')$. By Lemma \ref{varphi ind}, this implies that $\add(\Ind N)=\add(\Ind N')$. Thus 
\[
\add N=\add\Bigl(\bigoplus_{g\in G}\upsigma^*_gN\Bigr)=\add\Bigl(\bigoplus_{g\in G}\upsigma^*_gN'\Bigr)=\add N'.
\]
Hence $N=N'$ in the set $\mbox{$G$-}\modif S$.
\end{proof}
\end{cor}

A similar statement for MM modules might be more subtle, and we only prove the MM version in a special case.  For this, we need the following. 

\begin{prop}[{\cite[Proposition 5.11]{IW14a}}]\label{3d MM}
Assume that $S$ is 3-dimensional Gorenstein and admits an NCCR. For $N\in\refl S$, $\bLam_N$ is an NCCR of $S$ if and only if $N\in \catMM S$.
\end{prop}

The following is an MM version of Corollary \ref{modif inj}.

\begin{cor}\label{mm descent} 
Assume that the $G$-action on $S$ is small and  that $S$ is a 3-dimensional Gorenstein ring having an NCCR. Then the functor $\upvarphi_*\colon\fmod S\to\fmod R$ induces an injective map
\[
G\text{-}\modif S\cap\MM S\hookto \MM R.
\]
\begin{proof}
Let $(N,\theta)\in\catMM_GS$. By Proposition \ref{3d MM},  $\bLam_N$ is an NCCR of $S$, and so   $\bLam_{\upvarphi_*N}$ is an NCCR of $R$ by Theorem \ref{main 1}. Thus $\upvarphi_*N\in \catMM R$ by Lemma \ref{nccr mm}\,(1). 
\end{proof}
\end{cor}

\vspace{2mm}
\subsection{$G$-maximal modification algebras}\label{eq MMA section}~
\vspace{2mm}

In this section, we always assume that  the $G$-action on $S$ is small. Then we have a functor 
\[
\upvarphi_*\colon \catmodif_GS\to \catmodif R;\quad (N,\theta)\mapsto \upvarphi_*N
\]
by Theorem \ref{main 1}. We say that $\upvarphi_*\colon \catmodif_GS\to \catmodif R$ is {\it dense} if for any $M\in\catmodif R$, there exists $(N,\theta)\in\catmodif_GS$ such that $M$ is a direct summand of $\upvarphi_*N$. 

We consider a similar problem to the previous section for MM modules. To this end, we introduce the following.

\begin{dfn}
A  $G$-equivariantizable modifying $S$-module $N$ is {\it $G$-maximal modifying} ({\it $G$-MM}) if the  implication 
\[
N\oplus N'\in\catmodif S \quad\Longrightarrow\quad N'\in\add N
\]
holds for every $N'\in \mbox{$G$-}\modif S$.
We denote by 
\[
\mbox{$G$-}\MM S\quad(\subseteq \mbox{$G$-}\modif S)
\]
 the set of additive equivalence classes of $G$-maximal modifying $S$-modules.
\end{dfn}

\begin{rem}
It is obvious that a $G$-equivariantizable MM $S$-module is a $G$-MM $S$-module. However, the converse does not hold in general. Moreover, even if $S$ has a $G$-MM module, it could happen that $S$ has no $G$-equivariant MM module (see Example \ref{ex no nccr}). 
\end{rem}

\begin{lem}\label{direct summand lem}
Let $M\in\refl R$. Then $M$ is a direct summand of $\upvarphi_*((\upvarphi^*M)^{**})$. 
\begin{proof}
By Proposition \ref{inv refl}, $((\upvarphi^*M)^{**})^G\cong M$. Since $|G|$ is invertible in $S$, we have a Reynolds operator. Thus  $((\upvarphi^*M)^{**})^G$ is a direct summand of $\upvarphi_*((\upvarphi^*M)^{**})\in\refl R$. 
\end{proof}
\end{lem}

\begin{thm}\label{main 2}
Assume that $\upvarphi_*\colon \catmodif_GS\to \catmodif R$ is  dense. 
\begin{itemize}
\item[$(1)$] If $N\in\refl S$ is $G$-maximal modifying, then $\upvarphi_*N\in\catMM R$.
\item[$(2)$] For every $M\in\catMM R$, there is $N\in G\text{-}\MM S$ such that $\add(\upvarphi_*N)=\add M$. 
\item[$(3)$] For every $M\in\catmodif R$ giving an NCCR of $R$, there is $N\in G\text{-}\modif S$  such that $N$ gives an NCCR of $S$ and $\add(\upvarphi_*N)=\add M$.
\item[$(4)$] For every $M\in \catCT R$, there is $(N,\theta)\in \catCT_GS$ such that $\add(\upvarphi_*N)=\add M$. 
\end{itemize}
\begin{proof}
(1) Let $N\in G\text{-}\MM S$ and $M\in\catmodif R$ such that $\upvarphi_*N\oplus M\in\catmodif R$. By assumption, there is $N'\in G\text{-}\modif S$ such that $\upvarphi_*N\oplus M$ is a direct summand of $\upvarphi_*N'\in\catmodif R$.  Then $\bigl(\upvarphi^*(\upvarphi_*N\oplus M)\bigr)^{**}$ is a direct summand of $(\upvarphi^*\upvarphi_*N')^{**}$. By Lemma \ref{varphi ind}, we have $(\upvarphi^*\upvarphi_*N')^{**}\cong \Ind(N')\in\catmodif_GS$. Thus $(\upvarphi^*(\upvarphi_*N\oplus M))^{**}\cong \Ind(N)\oplus(\upvarphi^*M)^{**}$ is also modifying. In particular, $N\oplus(\upvarphi^*M)^{**}\in\catmodif S$. Since $N$ is $G$-MM, we have $(\upvarphi^*M)^{**}\in\add N$.  Thus, by Lemma \ref{direct summand lem},  $M\in \add\upvarphi_*((\upvarphi^*M)^{**})\subseteq\add\upvarphi_*N$. Hence $\upvarphi_*N\in\catMM R$.\vspace{1mm}\\
(2) Let $M\in \catMM R$. By assumption, there is a $G$-equivariantizable $N\in \catmodif S$ such that $M\in\add(\upvarphi_*N)$. Then $M\leq \upvarphi_*N$ in the poset $\modif R$. By Lemma \ref{poset mm},  we have $\add(M)=\add(\upvarphi_*N)$. It suffices to prove that $N$ is $G$-MM. Let $L\in G\text{-}\modif S$ such that $N\oplus L\in\catmodif S$. Then $\upvarphi_*(N\oplus L)\in\catmodif R$ by Theorem \ref{main 1}. Since $\upvarphi_*N$ is also MM, we have $\upvarphi_*L\in\add(\upvarphi_*N)$. Therefore, it follows from Lemma \ref{varphi ind} that $L\in\add(L)=\add\bigl((\upvarphi^*\upvarphi_*L)^{**}\bigr)\subseteq \add\bigl((\upvarphi^*\upvarphi_*N)^{**}\bigr)=\add N$. This proves that $N$ is $G$-MM. 
\vspace{1mm}\\
(3)  Let $M\in\refl R$, and assume that $\bLam_M$ is an NCCR of $R$. Then $M\in\catMM R$ by Lemma \ref{nccr mm}\,(1).  Take $(N,\theta)\in\catmodif_GS$ such that $M\in \add(\upvarphi_* N)$.   Then we have $\add M=\add(\upvarphi_*N)$. Since $\bLam_M$ and $\bLam_{\upvarphi_*N}$ are Morita equivalent, $\bLam_{\upvarphi_*N}$ is also an NCCR of $R$.  By Theorem \ref{main 1}\,(2), $N$ gives an NCCR of $S$. \vspace{1mm}\\
(4) Let $M\in\catCT R$. By Proposition \ref{char ct}, we have $M\in \catMM R$. Thus by (3), there is $(N,\theta)\in\catmodif_GS$ such that $N$ gives an NCCR of $S$ and $\add M=\add(\upvarphi_*N)$. Then $\upvarphi_*N$ is also a CT $R$-module, and so $(N,\theta)\in\catCT_G S$ by Theorem \ref{main 1}\,(3).
\end{proof}
\end{thm}

By Corollary \ref{modif inj} and Theorem \ref{main 2}, we obtain the following.

\begin{cor}\label{mm descent 2} 
Assume that $\upvarphi_*\colon \catmodif_GS\to \catmodif R$ is  dense. 
The functor $\upvarphi_*\colon\fmod S\to\fmod R$ induces a bijective map
\[
\upvarphi_*\colon\mbox{$G$-}\MM S\simto \MM R.
\]
If $S$ admits an NCCR, there are bijective maps
\[
\upvarphi_*\colon\mbox{$G$-}\NCCR S\simto \NCCR R\quad\mbox{and}\quad\upvarphi_*\colon\mbox{$G$-}\CT S\simto \CT R.
\]
\end{cor}

\begin{dfn}
We say that $S$ is {\it $G$-MM-finite} (resp. {\it $G$-MM-simple}) if $G\text{-}\MM S$ is a non-empty finite set (resp. has a unique element). 
\end{dfn}

The following immediately follows from Corollary \ref{mm descent 2}. 

\begin{cor}
Assume that  $\upvarphi_*\colon \catmodif_GS\to \catmodif R$ is  dense. Then $S$ is $G$-MM-finite (resp. $G$-MM-simple) if and only if $R$ is MM-finite (resp. MM-simple). 
\end{cor}

\vspace{2mm}
\section{Applications to canonical covers}

In this section, we apply the results  in the previous section to canonical covers.
Throughout this section,  $R$ is a Cohen--Macaulay normal domain containing a field $k$ of characteristic zero, and we assume that there is a canonical $R$-module $\omega_R$.

For $M,N\in\refl R$, we define the {\it reflexive product}
\[
M\rp N\defeq(M\otimes_RN)^{**}.
\]
The following is a summary of basic properties of reflexive products.

\begin{lem}
Let $L,M,N\in\refl R$.
\begin{itemize}
\item[$(1)$] $L\rp (M\rp N)\cong(L\rp M)\rp N$.
\item[$(2)$] $\Hom(L\rp M,N)\cong \Hom(L,\Hom(M,N))$.
\item[$(3)$] $\Hom(M,N)\cong M^*\rp N$.
\end{itemize}
\end{lem}

For $M\in\refl R$, we put $M^{[0]}\defeq R$ and, for $n>0$, we set
\[
M^{[n]}\defeq \underbrace{M\rp M\rp\,\,\cdots\,\,\rp M}_{n}.
\]
\begin{dfn}
 $R$ is {\it $\Q$-Gorenstein} (resp. {\it $\Q$-Calabi--Yau}, or {\it $\Q$-CY}) {\it of index $r$} if 
\[
r=\min\left\{\,n>0\,\,\middle|\,\, \omega_R^{[n]}\text{ is projective (resp. free) for some canonical module $\omega_R$}\right\}.
\]
\end{dfn}

\begin{rem}
If $R$ is local, $R$ is $\Q$-Gorenstein of index $r$ if and only if it is $\Q$-Calabi--Yau of index $r$.
\end{rem}

In the rest of this section, we assume that $R$ is $\Q$-CY of index $r$ and that $k$ contains a primitive $r$-th root of unity $\varepsilon$, and  we fix a canonical module $\omega_R$ such that $\omega_R^{[r]}\cong R$. Set $\upmu_r\defeq\l\varepsilon\r\subseteq k^{\times}$. We define an $R$-algebra
\[
S\defeq\bigoplus_{i=0}^{r-1}\omega^{[i]}_R,
\]
where the ring structure of $S$ is induced from  a choice of an isomorphism $\omega^{[r]}_R\cong R$.
The ring $S$ is called a {\it canonical cover} of $R$.  A canonical cover $S$ has a natural  $\upmu_r$-action defined by 
\[
\upsigma_{\varepsilon}\colon S\simto S;\quad \sum_{i=0}^{r-1} x_i\mapsto \sum_{i=0}^{r-1} \varepsilon^ix_i, \quad (x_i\in\omega^{[i]}_R).
\]
 Then $S^{\upmu_r}=R$.
\begin{prop}\label{prop:canonical cover}
We have the following:
\begin{itemize}
\item[$(1)$] $S$ is a normal domain.
\item[$(2)$] $S$ is Gorenstein if and only if $\omega^{[i]}_R\in\CM R$ for $1\leq i\leq r$.
\item[$(3)$] The $\upmu_r$-action on $S$ is small.
\end{itemize}
\begin{proof}
(1) and (2) follow from \cite[Proposition 2.15]{DITW}. \\
(3) Set $V\defeq \{ \p\in\Spec R\,|\,(\omega_R)_{\p}\in\proj R_{\p}\}$, and set $U\defeq \upvarphi^{-1}(V)\subseteq \Spec S$.  Note that $V$ is open in $\Spec R$. Then the sheaf $\omega_R|_V\cong\omega_V$ is invertible, and we have $(\omega_R^{[i]}\otimes\omega^{[j]}_R)|_V\cong (\omega^{[i+j]}_R)_V$. Hence,  the restriction $\upvarphi|_U\colon U\to V$ is a principal $\upmu_r$-bundle by \cite[Lemma 15.7]{Has}.
Since all height one prime ideals lie in $V$, we have $\codim(\Spec R\backslash V)\geq2$. Therefore, $\upvarphi$ is an almost principal $\upmu_r$-bundle, and thus the result follows from Lemma \ref{small principal}.
\end{proof}
\end{prop}

Denote by
\[
(-)^{\vee}\defeq\Hom(-,\omega_R)\colon \refl R\simto\refl R
\]
the canonical duality which restricts to a duality $(-)^{\vee}\colon\CM R\simto\CM R$. Composing this with the $R$-dual, we define the {\it Nakayama functor} 
\[
\upnu\defeq(-)^{\vee}\circ(-)^*\colon\refl R\simto \refl R.
\]
Note that there is an isomorphism $\upnu(M)\cong M\rp\omega_R$ for each $M\in\refl R$.

\begin{dfn} Here we only assume that $R$ is a Cohen--Macaulay normal domain. Let $M\in\catmodif R$.
\begin{itemize}
\item[(1)] For a family $\cF=\{F\colon \refl R\simto \refl R\}$  of  additive autoequivalences, we say that $M$ is {\it right (resp. left) $\cF$-stable} if 
\[
\Hom_R(M,F(M))\in\CM R \quad(\mbox{resp.}\quad \Hom(F(M),M)\in\CM R)\]
 for all $F\in\cF$. Moreover, $M$ is said to be {\it $\cF$-stable} if it is right and left $\cF$-stable.
\item[(2)]
Assume that $R$ is $\Q$-Gorenstein of index $r$, and set $\cN\defeq\{\upnu, \upnu^2,\hdots,\upnu^{r-1}\}$. Then, we say that $M$ is {\it Nakayama-stable} if it is $\cN$-stable. 
\end{itemize}
\end{dfn}

\begin{lem}\label{pull back modif}
Assume that $S$ is Gorenstein. Let $M\in\catmodif R$. The following are equivalent.
\begin{itemize}
\item[$(1)$] $M$ is Nakayama-stable. 
\item[$(2)$] $(\upvarphi^*M)^{**}\in\catmodif_{\upmu_r}S$.
\item[$(3)$] There exists $N\in\catmodif _{\upmu_r}S$ such that $M\in\add (\upvarphi_*N)$.
\end{itemize}
\begin{proof}
(1)$\Leftrightarrow$(2) For $(N,\theta)\defeq(\upvarphi^*M)^{**}\in\refl_{\upmu_r}S$, the condition (2) is equivalent to 
\[
\upvarphi_*N\in\catmodif R
\]
by Theorem \ref{main 1}. There are isomorphisms of $R$-modules
\[
\upvarphi_*N\cong \upvarphi_*\Hom_R(S,M)\cong \Hom_R(\upvarphi_*S,M)\cong\Hom_R\Bigl(\bigoplus_{i=0}^{r-1}\omega^{[i]}_R,M\Bigr)\cong \bigoplus_{i=0}^{r-1}\bigl(\omega^{[i]}_R\rp M\bigr).
\]
Thus $\upvarphi_*N\in\catmodif R$ if and only if $M$ is Nakayama-stable since $\upnu^i(M)\cong M\cdot \omega^{[i]}_R$.\\
(2)$\Rightarrow$(3) This follows from Lemma \ref{direct summand lem}.\\
(3)$\Rightarrow$(2)  Assume that there exists $N\in\catmodif _{\upmu_r}S$ such that $M\in\add (\upvarphi_*N)$. Then $(\upvarphi^*M)^{**}\in\add (\upvarphi^*\upvarphi_*N)^{**}=\add N$, where the equality follows from Lemma \ref{varphi ind}. Since $N$ is modifying, so is $(\upvarphi^*M)^{**}$. 
\end{proof}
\end{lem}

\begin{thm}\label{mm r cover}
 Assume that $S$ is Gorenstein. If every $M\in\catmodif R$ is Nakayama-stable, there is a bijection
\[
\upvarphi_*\colon \upmu_r\text{-}\MM S\simto \MM R.
\]
\begin{proof}
If every $M\in\catmodif R$ is Nakayama-stable, the functor $\upvarphi_*\colon \catmodif_{\upmu_r}S\to\catmodif R$ is dense by Lemma  \ref{pull back modif}. Hence the result follows from Corollary \ref{mm descent 2}. 
\end{proof}
\end{thm}

\begin{lem}\label{dense lem}
Assume that $r=2$. Then every $M\in\catmodif R$ is Nakayama-stable. 
\begin{proof}
Let $M\in\catmodif R$. By definition, $\Hom(M,M)\in\CM R$. Since 
\[\Hom(M,M\rp\omega_R)\cong\Hom(M,M)\rp\omega_R\cong(\Hom(M,M)^*)^{\vee}\cong \Hom(M,M)^{\vee},\] 
we have $\Hom(M,M\rp\omega_R)\in\CM R$. Since $\Hom(M\rp\omega_R)\cong \Hom(M,M\rp\omega_R)\in\CM R$, we see that $M$ is Nakayama-stable.  
\end{proof}
\end{lem}

\begin{thm}\label{index two}
Assume that $r=2$. Then the functor 
$
\upvarphi_*\colon\catmodif_{\upmu_2}S\to\catmodif R
$  induces a bijective map
\[
\upvarphi_*\colon\mbox{$\upmu_2$-}\MM S\simto \MM R.
\]
\begin{proof}
Since $\omega_R\in\CM R$, the canonical cover $S$ is Gorenstein by Proposition \ref{prop:canonical cover}\,(2). Hence, the result follows from  Lemma \ref{dense lem} and Theorem \ref{mm r cover}.
\end{proof}
\end{thm}

As applications of Theorem \ref{index two}, we obtain the following corollaries.

\begin{cor}\label{mu2 prop}
Assume that $r=2$. Then $R$ is MM-finite (resp. MM-simple) if and only if $S$ is $\upmu_2$-MM-finite (resp. $\upmu_2$-MM-simple). 
\end{cor}

\begin{cor}\label{mm fact}
Assume that $r=2$ and  one of the following conditions holds.
\begin{itemize}
\item[$(1)$] $S$ is a regular local ring.
\item[$(2)$] $S$ is a three-dimensional  $\Q$-factorial local hypersurface isolated singularity. 
\end{itemize}
Then $\upvarphi_*S$ is the unique MM $R$-module up to additive equivalence. In particular, $R$ is MM-simple. 
\begin{proof}
If (1) is satisfied, we have $\catmodif S=\proj S$ by \cite[Theorem 4.4]{AG60}. If (2) holds,  we have $\catmodif S=\proj S$ by results in \cite{Dao} (see also \cite[Theorem 2.10]{IW14b}). Hence $\upmu_2\text{-}\MM S=\{S\}$, and then the assertion follows from Theorem \ref{index two}. 
\end{proof}
\end{cor}

The following gives the first  example of a non-Gorenstein MM-simple normal domain that admits no NCCR. 

\begin{ex}
Let $m\in\Z_{\geq1}$, $S\defeq\C[\![x,y,z,w]\!]/\l xy-z^2-w^{2m+1}\r$ and $G\defeq\frac{1}{2}(1,1,1,0)$. Then 
\begin{align*}
S^G&=\C[\![x^2,y^2,z^2,xy,yz,zx,w]\!]/\l xy-z^2-w^{2m+1}\r\\
&\cong \frac{\C[\![X,Y,Z,T,U,V,W]\!]}{\l T-Z-W^{2m+1}, XY-T^2,YZ-U^2,XZ-V^2,XU-TV,YV-TU,ZT-UV\r}\\
&\cong\C[\![X,Y,Z,U,V,W]\!]/I_2(A),
\end{align*}
where $I_2(A)$ denotes the ideal generated by $2\times2$ minors  of the symmetric matrix given by
\[
A\defeq
\begin{pmatrix}
X&Z+W^{2m+1}&V\\
Z+W^{2m+1}&Y&U\\
V&U&Z
\end{pmatrix}.
\]
The invariant ring $R\defeq S^G$ is $\Q$-CY of index two, and $S$ is the canonical cover of $R$. Since $S$ is $\Q$-factorial, $R$ has a unique basic MM $R$-module $\upvarphi_*S\cong R\oplus \omega_R$ by Corollary \ref{mm fact}. 
\end{ex}

The following example illustrates an advantage of considering $G$-MM modules rather than $G$-equivariant MM modules.

\begin{ex}\label{ex no nccr}
Let $S\defeq \C[\![x,y,z,w]\!]/\l xy-zw\r$, and consider the small $\upmu_2$-action on $S$ defined  by
\[
(-1)\cdot x=-x, \quad(-1)\cdot y=-y, \quad(-1)\cdot z=w, \quad(-1)\cdot w=z.
\]
Then $R\defeq S^{\upmu_2}$ is $\Q$-CY of index two, and  $S$ is a canonical cover of $R$. Let $L\defeq \l x,z\r\subset S$, which is a generator of $\Cl S\cong \Z$. Then a basic MM $S$-module is of the form $N_{k}\defeq L^{[k]}\oplus L^{[k+1]}$ for some $k\in\Z$ (see \cite[Theorem 5.20]{IW18} for a more general result).  If $L^{[k]}$ is $\upmu_2$-equivariantizable, there is an $S$-module isomorphism $L^{[k]}\cong\upsigma_{(-1)}^*(L^{[k]})\cong L^{[-k]}$. Thus $L^{[k]}$ is $\upmu_2$-equivariantizable if and only if $k=0$. In particular, $\upmu_2\text{-}\MM S=\{S\}$. Hence, $S$ is not MM-finite but $\upmu_2$-MM-simple.  By Corollary \ref{mu2 prop}, $R$ is MM-simple and $\MM R=\{\upvarphi_*S\}$. 
Note that $\upvarphi_*S$ does not give an NCCR of $R$, since $S$ is not an MM $S$-module. Actually, although $S$ admits an NCCR, the quotient $R$ admits no NCCR by Corollary \ref{mm descent 2}.
 \end{ex}

\vspace{2mm}
\section{Three-dimensional terminal singularities of index two}
 
 In this section, by using results in the previous section, we prove Conjecture \ref{intro conj} for terminal singularities of index two and Conjecture \ref{intro conj 2} for $\Q$-Gorenstein rings of index two. For the latter, we generalize derived equivalences of MMAs over Gorenstein rings \cite{IW14a}, to derived equivalences of Gorenstein MMAs over $\Q$-Gorenstein rings.

  \vspace{2mm}
\subsection{Derived equivalences of Gorenstein MMAs}~
\vspace{2mm}

\vspace{2mm}
\subsubsection{Tilting modules}~
\vspace{2mm}

Let $R$ be a noetherian normal domain, and let $\Lambda$ be a module-finite $R$-algebra. 
First, we recall the definition of a (partial) tilting module.

\begin{dfn}\label{tilt def}
We say that $T\in \fmod \Lambda$ is {\it tilting} if the following conditions are satisfied.
\begin{itemize}
\item[(1)] $\prdim_{\Lambda} T\leq 1$.
\item[(2)] $\Ext^1_{\Lambda}(T,T)=0$.
\item[(3)] There exists an exact sequence
\begin{equation}\label{tilt seq}
0\to \Lambda\to T^0\to T^1\to 0
\end{equation}
with $T^0,T^1\in \add T$.
\end{itemize}
We say that $T$ is {\it partial tilting} if the first two conditions are satisfied. 
\end{dfn}

We call an object $A$ of a triangulated category $\scrT$  a {\it split generator} if the smallest thick subcategory of $\scrT$ containing $A$ is the whole category $\scrT$.

\begin{prop}[{\cite[Proposition 7.4]{NS}}]\label{tilting dfn}
Let $T\in\fmod \Lambda$ be a partial tilting $\Lambda$-module. Then $T$ is tilting if and only if $T$ is a split generator of  $\Kb(\proj\Lambda)$. 
\end{prop}

By Proposition \ref{tilting dfn}, a module $T\in\fmod\Lambda$ with $\prdim T\leq1$ is tilting if and only if it is a tilting complex in the usual sense. Thus we have the following.

\begin{lem}[{\cite[Remark 4.8]{GIK}}]
For $T\in\fmod \Lambda$, the following are equivalent:
\begin{itemize}
\item[$(1)$] $T$ is a tilting $\Lambda$-module. 
\item[$(2)$]$T_{\p}$ is a tilting $\Lambda_{\p}$-module for all $\p\in\Spec R$. 
\item[$(3)$] $T_{\m}$ is a tilting $\Lambda_{\m}$-module for all maximal ideals $\m$ of $R$.
\end{itemize}
\end{lem}

\vspace{2mm}
\subsubsection{Derived equivalences of Gorenstein MMAs}~
\vspace{2mm}

Throughout this section, $R$ is a three-dimensional $\Q$-Gorenstein Cohen--Macaulay equi-codimensional normal domain of index $r$. Denote by $\omega_R$ a canonical $R$-module. 

We begin by giving the following easy observation. 

\begin{lem}\label{add lem}
For any $M\in\refl R$, we have $\add(M)=\add(M\cdot \omega_R^{[r]})$.
\begin{proof}
Since the index of $R$ is $r$, we have $\omega_R^{[r]}\in\proj R$. Since $R$ is normal, $\Supp (\omega^{[r]}_R)=\Spec R$.  Thus $(\omega_R^{[r]})_{\p}$ is a non-zero free $R_{\p}$-module for all $\p\in\Spec R$. Thus $\add(M_{\p})=\add(M_{\p}\cdot(\omega_R^{[r]})_{\p})=\add((M\cdot\omega_R^{[r]})_{\p})$ for all $\p\in\Spec R$. By \cite[Proposition 2.26]{IW14a}, we obtain $\add(M)=\add(M\cdot\omega^{[r]}_R)$. 
\end{proof}
\end{lem}

\begin{dfn}
Let $M\in\catmodif R$. The modification $R$-algebra $\bLam_M$ is {\it Gorenstein} if $\injdim_{\bLam_{M_{\p}}}\bLam_{M_{\p}}=\dim R_{\p}$ for all $\p\in\Spec R$.
\end{dfn}

For $M\in\catmodif R$, we write
\[
\omega_{\bLam_M}\defeq \Hom(\bLam_M,\omega_R)\in \CM \bLam_M. 
\]
By the canonical duality $(-)^{\vee}\colon \CM \bLam_M\to\CM \bLam^{\rm op}_M$, the Cohen--Macaulay module $\omega_{\bLam_M}$ is an injective object in $\CM \bLam_M$. In particular, 
\begin{equation}\label{eq:omega inj}
\Ext^i(X,\omega_{\bLam_M})=0\quad\text{for $i>0$ and $X\in\CM \bLam_M$}
\end{equation}
\begin{rem}\label{gor def}
By \cite[Lemma 3.2]{DITW}, the above definition of Gorenstein modification is equivalent to the definition given in \cite[Definition 2.14]{IW14a}. Namely, $\bLam_M$ is Gorenstein in the above sense  if and only if $\omega_{\bLam_M}\in\proj\bLam_M$, which is equivalent to $\add(\bLam_M)=\add(\omega_{\bLam_M})$ by \cite[Lemma 2.15]{IW14a}. 
\end{rem}

The following is a characterization of Gorenstein MMAs.

\begin{lem}\label{GM MMA}
Let $M\in\catMM R$. Then the following are equivalent.
\begin{itemize}
\item[$(1)$] $\bLam_{M}$ is  Gorenstein.
\item[$(2)$] $\add M=\add(M\cdot \omega_R)$
\item[$(3)$] $\add M=\add(M\cdot \omega_R^{[i]})$ for any $1\leq i\leq r$.
\item[$(4)$] $M$ is Nakayama-stable.
\item[$(5)$] $\bLam_{M}^{\rm op}$ is Gorenstein.
\end{itemize}
\begin{proof}
(1)$\Longleftrightarrow$(2) is \cite[Proposition 3.4]{DITW}, and (1)$\Longleftrightarrow$(5) is \cite[Lemma 3.2]{DITW}.  (3)$\Rightarrow$(2) is obvious. (2)$\Rightarrow$(3) follows from an equality $\add (M\cdot\omega^{[i]}_R)=\add(M\cdot \omega^{[i+1]}_R)$ for each $0\leq i\leq r-1$, which follows since $\bLam_{M\cdot\omega^{[i]}_R}\cong\bLam_M$ is also Gorenstein. Thus, (1)$\Longleftrightarrow$(2)$\Longleftrightarrow$(3)$\Longleftrightarrow$(5).\\  
(3)$\Rightarrow$(4) If we have $\add M=\add(M\cdot\omega_R^{[i]})$ for $1\leq \forall i\leq r$, then $\Hom(M,M\cdot \omega_R^{[i]})\in \CM R$ and $\Hom(M\cdot \omega_R^{[i]},M)\in \CM R$.  Thus $M$ is Nakayama-stable. \\
(4)$\Rightarrow$(2)  If $M$ is Nakayama-stable, then $\Hom(M,M\cdot\omega_R)$ and  $\Hom(M\cdot\omega_R,M)$ are both Cohen--Macaulay. Hence, $M\oplus (M\cdot\omega_R)\in\catmodif R$. Since $M$ and $M\cdot\omega_R$ are maximal modifying, we obtain $\add M=\add (M\cdot\omega_R)$. 
\end{proof}
\end{lem}

\begin{dfn}
Let $M,N\in\refl R$. A morphism $\alpha\colon M'\to N$ is said to be a {\it $(\add M)$-approximation} if $M'\in\add M$ and  $\alpha\circ (-)\colon\Hom(M,M')\to\Hom(M,N)$ is surjective.
\end{dfn}

The following is our key generalization of \cite[(1)$\Rightarrow$(2) in Theorem 4.12]{IW14a} to a $\Q$-Gorenstein case.

\begin{lem}\label{GM lemma}
Let $M\in\catMM R$, and let $N\in\catmodif R$. Assume that $\bLam_M$ is Gorenstein, and consider an exact sequence
\[
0\to L\xrightarrow{\beta}  M'\xrightarrow{\alpha} N,
\]
where $\alpha$ is an $(\add M)$-approximation. Then $L\in \add M$.
\begin{proof}
Since $M\in\catMM R$, it is enough to show that $M\oplus L\in \catmodif R$. Write 
\[
\sfF(-)\defeq \Hom(M,-)\colon \refl R\simto\refl\bLam_M.
\] 
Then
\begin{equation}\label{eq:6.4}
0\to \sfF(L)\xrightarrow{\sfF(\beta)}  \sfF(M')\xrightarrow{\sfF(\alpha)} \sfF(N)\to 0
\end{equation} 
is exact by assumption. Since $\sfF(M')\in\CM R$,  $\sfF(N)\in \refl R$ and $\dim R=3$, the depth lemma implies that $\Hom(M,L)=\sfF(L)\in\CM R$. By Lemma \ref{GM MMA}, the $R$-algebras $\bLam_M$ and $\bLam_{M\cdot \omega_R^{[r-1]}}$ are Morita equivalent. Put $\sfG(-)\defeq \Hom(M\cdot\omega_R^{[r-1]},-)\colon \refl R\simto \refl\bLam_{M\cdot\omega_R^{[r-1]}}$. Then the functors $\sfF$ and $\sfG$ can be identified via Morita equivalence of $\bLam_M$ and $\bLam_{M\cdot \omega^{[r-1]}_R}$, and so  the sequence $0\to \sfG(L)\xrightarrow{\sfG(\beta)}  \sfG(M')\xrightarrow{\sfG(\alpha)} \sfG(N)\to 0$ is also exact. Since $M$ is Nakayama-stable by Lemma \ref{GM MMA}, we have $\sfG(M')\in\CM R$. Hence, by the same argument, we obtain $\Hom(M\cdot \omega_R^{[r-1]},L)=\sfG(L)\in\CM R$. By Lemma \ref{add lem},
\[
\Hom(L,M)\in\add\left(  \Hom(L,M\cdot\omega^{[r-1]}_R)\cdot\omega_R\right)=\add\left( \Hom(M\cdot\omega_R^{[r-1]},L)^{\vee}\right).
\]
Since the canonical dual preserves Cohen--Macaulay modules, we have $\Hom(L,M)\in\CM R$. 

Now it remains to prove that  $\End(L)\in\CM R$. By a reflexive equivalence, there is an $R$-module isomorphism $\End(L)\cong \End(\sfF(L))$. Hence, by \cite[Lemma 2.7]{IW14a}, it suffices to show that 
$\fl\Ext^1_{\bLam_M}(\sfF(L),\sfF(L))=0$, where $\fl(-)$ denotes the finite length submodule of $(-)$. The vanishing  $\fl\Ext^1_{\bLam_M}(\sfF(L),\sfF(L))=0$ follows from the identical argument as in the latter part of the proof of  \cite[Lemma 4.10]{IW14a}. We repeat it for the reader's convenience. 

Consider the following commutative diagram
 \[
\begin{tikzcd}
\Hom(\sfF(M'),\sfF(M'))\arrow[rr, "f\defeq\sfF(\alpha)\circ(-)"]\arrow[d, "a\defeq(-)\circ \sfF(\beta)"']&&\Hom(\sfF(M'),\sfF(N))\arrow[d, "b\defeq(-)\circ \sfF(\beta)"]\\
\Hom(\sfF(L),\sfF(M'))\arrow[rr, "g\defeq\sfF(\alpha)\circ(-)"]&&\Hom(\sfF(L),\sfF(N)).
\end{tikzcd} 
\]
Applying $\Hom(\sfF(L),-)$ to the exact sequence \eqref{eq:6.4}, we obtain a morphism 
\[
\gamma\colon \Ext^1(\sfF(L),\sfF(L))\to \Ext^1(\sfF(L),\sfF(M'))
\] 
and an exact sequence
\begin{equation}\label{eq:6.4 2}
0\to \Cok(g)\to \Ext^1(\sfF(L),\sfF(L))\to \Im(\gamma)\to 0.
\end{equation}
Since $\Hom(\sfF(L),\sfF(M'))\cong\Hom(L,M')\in\CM R$ and $\sfF(M')\in\CM R$, we have 
\[\fl\Ext^1(\sfF(L),\sfF(M'))=0\] 
by \cite[Lemma 2.7]{IW14a}. In particular, $\fl \Im(\gamma)=0$. By \eqref{eq:6.4 2}, it is enough to show that $\fl \Cok(g)=0$.  There is an exact sequence
\begin{equation}\label{eq: 6.4 3}
\Cok(a)\to \Cok(g\circ a)\to\Cok(g)\to 0.
\end{equation}
Since $\sfF(M')\in\proj \bLam_M$, the map $f$ is surjective. Thus $\Cok(g\circ a)=\Cok(b\circ f)=\Cok(b)$. Since $\Cok(b)\subseteq \Ext^1(\sfF(N),\sfF(N))$ and $\fl\Ext^1(\sfF(N),\sfF(N))=0$ by \cite[Lemma 2.7]{IW14a}, we see that $\fl\Cok(g\circ a)=\fl\Cok(b)=0$. For $\p\in\Spec R$ with $\text{ht}(\p)=2$, we have $\sfF(N)_{\p}\in\CM \bLam_{M_{\p}}$,  and by Remark \ref{gor def}, we have $\sfF(M')_{\p}\in\proj\bLam_{M_{\p}}=\add(\omega_{\bLam_{M_{\p}}})$ since $\bLam_{M_{\p}}$ is Gorenstein. Hence, by \eqref{eq:omega inj}, we have $\Ext^1(\sfF(N),\sfF(M'))_{\p}\cong \Ext^1(\sfF(N)_{\p},\sfF(M')_{\p})=0$ for $\p\in\Spec R$ with $\text{ht}(\p)=2$, and so $\Ext^1(\sfF(N),\sfF(M'))$ has finite length. 
Since $\Cok(a)\subseteq \Ext^1(\sfF(N),\sfF(M'))$,  the module $\Cok(a)$ also has finite length, and so the first map in \eqref{eq: 6.4 3} is the zero map. Thus $\fl\Cok(g)\cong\fl\Cok(g\circ a)=0$. 
\end{proof}
\end{lem}

The following gives a proof of a version of three-dimensional noncommutative Bondal--Orlov conjecture. Due to Lemma \ref{GM lemma}, we can prove it by a similar argument as in \cite[Proposition 4.14]{IW14a}. For the reader's convenience, we give a proof. 

\begin{thm}\label{Q-CY tilt}
Let $M\in\catMM R$, and assume that $\bLam_M$ is Gorenstein. Let $N\in\catmodif R$, and set $T\defeq\Hom(M,N)\in\refl \bLam_M$. 
\begin{itemize}
\item[$(1)$] The $\bLam_M$-module $T$ is partial tilting. 
\item[$(2)$] If $N\in\catMM R$ with $\bLam_N$ Gorenstein, then $T$ is a tilting $\bLam_M$-module. 
\end{itemize}
In particular, all Gorenstein MMAs are derived equivalent to each other. 
\begin{proof}
(1) By Lemma \ref{GM lemma}, we have $\prdim T\leq 1$. Let $\p\in\Spec R$ with $\text{ht}(\p)=2$. Then $T_{\p}\cong \Hom_{R_{\p}}(M_{\p},N_{\p})\in\refl R_{\p}=\CM R_{\p}$. Since the modification  $R_{\p}$-algebra $(\bLam_{M})_{\p}\cong \bLam_{M_{\p}}$ is Gorenstein and $\prdim T_{\p}\leq 1$, the CM $R_{\p}$-module $T_{\p}$ is a projective $\bLam_{M_{\p}}$-module by Auslander--Buchsbaum formula \cite[Lemma 2.16]{IW14a}. 
Hence $\Ext^1_{\bLam_M}(T,T)_{\p}\cong\Ext^1_{\bLam_{M_{\p}}}(T_{\p},T_{\p})=0$, and so it suffices to show that $\Ext^1_{\bLam_M}(T,T)_{\m}\cong \Ext^1_{\bLam_{M_{\m}}}(T_{\m},T_{\m})=0$ for all maximal ideals $\m$ of $R$. 

By the above argument,  we have  $\Ext^1_{\bLam_{M_{\m}}}(T_{\m},T_{\m})=\fl\Ext^1_{\bLam_{M_{\m}}}(T_{\m},T_{\m})$. Since $\bLam_{M_{\m}}$ is Gorenstein and $\Hom_{\bLam_{M_{\m}}}(T_{\m},T_{\m})\cong \Hom_{R_{\m}}(N_{\m},N_{\m})\in\CM R_{\m}$, we obtain the vanishing $\fl\Ext^1_{\bLam_{M_{\m}}}(T_{\m},T_{\m})=0$ by \cite[Lemma 2.7]{IW14a}.\\
(2) By (1), we only need to show that $T$ satisfies the condition (3) in Definition \ref{tilt def}. Let  $N\in\catMM R$ such that $\bLam_N$ is Gorenstein. Then $N^*\in\catMM R$. By Lemma \ref{GM MMA}, the MMA $\bLam_{N^*}\cong \bLam^{\rm op}_N$ is also Gorenstein. By Lemma \ref{GM lemma}, there are $N_0, N_1\in \add N$ and an exact sequence 
$
0\to N^*_1\to N^*_0\to M^*
$
such that the induced sequence
\begin{equation}\label{eq 6.10}
0\to \sfF(N^*_1)\to \sfF(N^*_0)\to \sfF(M^*)\to 0
\end{equation} 
is exact, where $\sfF(-)\defeq\Hom(N^*,-)$. Since $\sfF(M^*)$ is  partial tilting over $\bLam_{N^*}$ by (1), application of $\Hom(-,\sfF(M^*))$ to the sequence \eqref{eq 6.10} gives the following exact sequence
\[
0\to \Hom(\sfF(M^*),\sfF(M^*))\to\Hom(\sfF(N^*_0),\sfF(M^*))\to\Hom(\sfF(N_1^*),\sfF(M^*))\to0.
\]
Since $\sfF(-)\circ (-)^*\colon \refl R\to \refl \bLam_{N^*}$ is a duality, this exact sequence can be rephrased as an exact sequence
\[
0\to \bLam_M\to\Hom(M,N_0)\to\Hom(M,N_1)\to0.
\]
Since $\Hom(M,N_i)\in\add T$, this proves the condition (3) in Definition \ref{tilt def}. 
\end{proof}
\end{thm}

The following proves Conjecture \ref{intro conj 2} for $\Q$-Gorenstein rings of index two. 

\begin{cor}\label{goren mma}
Let $R$ be a three-dimensional Cohen--Macaulay equi-codimensional normal domain. If $R$ is $\Q$-Gorenstein of index two, then all MMAs are Gorenstein, and they are derived equivalent to each other.
\begin{proof}
Let $M$ be an MM $R$-module. By a similar argument as in the proof of Lemma \ref{dense lem}, we have $\Hom(M,M\cdot \omega_R)\in\CM R$. Moreover, 
\[
\Hom(M\cdot\omega_R,M)\cong \Hom(M\cdot\omega^{[2]}_R,M\cdot \omega_R)\in\add(\Hom(M,M\cdot \omega_R))\subseteq \CM R,
\]
where we use $\add(M)=\add(M\cdot \omega_R^{[2]})$ from Lemma \ref{add lem}. Thus, $M$ is Nakayama-stable. 
By  Lemma \ref{GM MMA}, the MMA $\bLam_M$ is Gorenstein, and Gorenstein MMAs are derived equivalent by  Theorem \ref{Q-CY tilt}\,(2). 
\end{proof}
\end{cor}

The following application of Theorem \ref{Q-CY tilt} will be necessary later.

\begin{cor}\label{ascending stop}
Assume that $R$ is complete local and admits an MM module $M$ such that $\bLam_M$ is Gorenstein. Denote by $|M|$ the number of indecomposable direct summands of $M$ that are non-isomorphic to each other. Then for any $N\in\catmodif R$, we have $|N|\leq |M|$. In particular, if there is an ascending chain
\[
 N^1<N^2<\cdots<N^{\ell}
\]
in the poset $\modif R$, then $\ell\leq |M|$. 
\begin{proof}
Let $N\in\catmodif R$. By Theorem \ref{Q-CY tilt}\,(1), the $\bLam_M$-module  $\Hom(M,N)$ is  partial tilting. Hence,  by \cite[Lemma 4.13]{IW14a}, there is $L\in\refl R$ such that $T\defeq\Hom(M,L)$ is a tilting $\bLam_M$-module and $N\in\add L$. Then $\bLam_M$ is derived equivalent to $\End_{\bLam_M}(T)\cong \bLam_L$. Since $R$ is complete local, $\fmod R$ is Krull--Schmidt, and $|N|<\infty$. We have 
\[
|N|\leq |L|=\rk K_0(\Kb(\proj\bLam_L))=\rk K_0(\Kb(\proj\bLam_M))=|M|.
\]
This proves the first assertion. The second statement follows from the first one, since an inequality $N^i<N^{i+1}$ in $\modif R$ implies an inequality $|N^i|<|N^{i+1}|$. 
\end{proof}
\end{cor}

 \vspace{2mm}
\subsection{Existence of MMAs: Complete local case}~
\vspace{2mm}

In this subsection, we prove Conjecture \ref{intro conj}\,(2)  for terminal singularities of index two. Let $R$ be a three-dimensional complete local $\C$-algebra with the residue field equal to $\C$. Assume that $R$ has a terminal singularity, and the index $r$ of $R$ is greater than one. Then the canonical cover $S$ of $R$ is Gorenstein and has at worst terminal singularities. If $S$ is regular, then $R$ has only a  quotient singularity, and so an MMA exists in this case. If $S$ is singular, it is isomorphic to an isolated cDV singularity  $\C[\![x,y,z,t]\!]/f$ for some $f\in\C[x,y,z,t]$. By Mori \cite[Theorem 12, 23, 25]{mor} (see also \cite[Theorem 8.3.18]{Ish}), (a standard form of) such a polynomial $f$ and the natural $\upmu_r$-action on the canonical cover $S=\C[\![x,y,z,t]\!]/f$ are classified into six families as in Figure \ref{Figure 1}. 
The $\upmu_r$-weights of polynomial $f$ in Figure \ref{Figure 1} are assumed to be zero except for the type $cAx/4$, and the weight of $f$ in the $cAx/4$ case is supposed to be two. Conversely, if $f$ and $\upmu_r$-action are taken from Figure \ref{Figure 1} and $(\C[\![x,y,z,t]\!]/f)^{\upmu_r}$ has only an isolated singularity, then it has a terminal singularity by \cite{KSB}.

\begin{figure}
\begin{center}
\renewcommand{\arraystretch}{1.65}
\setlength{\tabcolsep}{6pt}
\setlength{\arrayrulewidth}{0.5pt}

\resizebox{\textwidth}{!}{%
\begin{tabular}{
  |c|
  l|
  >{$\displaystyle}l<{$}|
  c
  >{$\displaystyle}l<{$}|
  >{$\displaystyle}l<{$}|
}
\hline
  \text{Type}  & \text{ Index} & \upmu_r\text{-action} & $f$&  & \text{Condition} \\
\hline
$cA/r$ & $r\geq 2$
    & \dfrac{1}{r}(a,-a,1,0)
    &
    & xy+g(z^r,t)
    & g\in\mathfrak{m}^{2},\ (a,r)=1
    \\[4pt]
\hline
\multirow{2}{*}{$cAx/4$} & \multirow{2}{*}{4}
    & \multirow{2}{*}{$\dfrac{1}{4}(1,1,3,2)$}
    &
    & xy+z^{2}+g(t)
    & g\in\mathfrak{m}^{3}
    \\[4pt]
    & & & or
    & x^{2}+z^{2}+g(y,t)
    & g\in\mathfrak{m}^{3}
    \\
\hline
$cAx/2$ & 2
    & \dfrac{1}{2}(0,1,1,1)
    &
    & x^{2}+y^{2}+g(z,t)
    & g\in\mathfrak{m}^{4}
    \\[4pt]
\hline
 & 
    & 
    &
    & x^{2}+y^{3}+z^{3}+t^{3}
    &
    \\[4pt]

  $cD/3$  & 3& \dfrac{1}{3}(0,2,1,1)& or
    & x^{2}+y^{3}+z^{2}t+yg(z,t)+h(z,t)
    & g\in\mathfrak{m}^{4},\ h\in\mathfrak{m}^{6}
    \\

    & & & or
    & x^{2}+y^{3}+z^{3}+yg(z,t)+h(z,t)
    & g\in\mathfrak{m}^{4},\ h\in\mathfrak{m}^{6}
    \\
\hline
 & 
    & 
    &
    & x^{2}+y^{3}+yzt+g(z,t)
    & g\in\mathfrak{m}^{4}
    \\[4pt]

   $cD/2$ & 2& \dfrac{1}{2}(1,0,1,1)& or
    & x^{2}+yzt+y^{n}+g(z,t)
    & g\in\mathfrak{m}^{4},\ n\geq 4
    \\

    & & & or
    & x^{2}+yz^{2}+y^{n}+g(z,t)
    & g\in\mathfrak{m}^{4},\ n\geq 3
    \\
\hline
$cE/2$ & 2
    & \dfrac{1}{2}(1,0,1,1)
    &
    & x^{2}+y^{3}+yg(z,t)+h(z,t)
    & g,h\in\mathfrak{m}^{4},\ h\not\in \mathfrak{m}^{5}.
    \\[4pt]
\hline
\end{tabular}%
}
\end{center}
\caption{}\label{Figure 1}
\end{figure}

The following proves Conjecture \ref{intro conj}\,(2) for terminal singularities of index two. 

\begin{cor}\label{MM ex local}
Let  $\Spec R$ be  a normal affine 3-fold over $\C$ with only terminal singularities of type $cA/2$, $cAx/2$, $cD/2$ or $cE/2$. Then $\widehat{R_{\m}}$ admits an MMA for all maximal ideals $\m$. 

\begin{proof}
We may assume that $\widehat{R_{\m}}$ is  non-Gorenstein. Then, by assumption, the index of $\widehat{R_{\m}}$ is two. Denote by $\widehat{S}$ the canonical cover of $\widehat{R_{\m}}$. Then $\widehat{S}$ is a complete local isolated cDV singularity, and $\bigl(\widehat{S}\bigr)^{\upmu_2}=\widehat{R_{\m}}$. 
 By Theorem \ref{index two}, it is enough to show that there exists a $\upmu_2$-MM $\widehat{S}$-module.

Since  $\widehat{S}$ is a Gorenstein terminal singularity, it has an MM module $M$ \cite[Theorem 1.9]{IW14b}  (see also \cite[Proposition 0.11]{IWtits}).  Since $\widehat{S}$ is Gorenstein, the MMA $\bLam_M$ is automatically Gorenstein. 
Let $N$ be a $\upmu_2$-equivariantizable modifying $\widehat{S}$-module (e.g. $N$ can be taken as $\widehat{S}$).  If $N$ is not $\upmu_2$-MM, we can take a $\upmu_2$-equivariantizable modifying module $N'$ with $N<N'$ in $\modif \widehat{S}$. By Corollary \ref{ascending stop}, this process terminates, and so we obtain a $\upmu_2$-MM module. 
\end{proof}
\end{cor}

 \vspace{2mm}
\subsection{Existence of MMAs: Global case}~
\vspace{2mm}

 Finally, we globalize Corollary \ref{MM ex local}.  To this end, we prove the following local-to-global property of the existence of MMAs.

\begin{prop}\label{local to global} 
Let $R$ be a three-dimensional $\Q$-Gorenstein excellent Cohen--Macaulay equi-codimensional normal domain of index $r$.
Assume the following conditions:
\begin{itemize}
\item[$(1)$] $R$ has only finitely many singular points $\m_1,\hdots,\m_n$. 
\item[$(2)$] For each $1\leq i\leq n$, there is an MM $\widehat{R_{\m_i}}$-module $M^i$ with $\bLam_{M^i}$ Gorenstein.
\end{itemize}
Then, if there is an ascending chain
\[
N^1<N^2<\cdots<N^{\ell}
\]
in the poset $\modif R$, then $\ell\leq 1+\sum_{k=1}^n(|M^k|-1)$. In particular, $R$ admits an MMA. 
\begin{proof}
Consider the inequality $N^i<N^{i+1}$ for each $1\leq i\leq \ell-1$. By \cite[Proposition 2.26]{IW14a}, there is $\p\in\Spec R$ such that $\widehat{N^i_{\p}}<\widehat{N^{i+1}_{\p}}$.  Since each $\widehat{N^i}$ is modifying, if $\widehat{R_{\q}}$ is regular, then $\add \widehat{N^i_{\q}}=\add \widehat{R_{\q}}=\add \widehat{N^{i+1}_{\q}}$. Thus $\p=\m_{j}$ for some $1\leq j\leq n$. For each $1\leq k\leq n$, set
 \[
 I_k\defeq\left\{1\leq j\leq \ell-1 \,\middle|\, \widehat{N^{j}_{\m_k}}<\widehat{N^{j+1}_{\m_k}}\right\}.
 \]
 If we write $n_k\defeq\# I_k$ and $I_k=\{j_1<j_2<\cdots<j_{n_k}\}$, there is an ascending chain
 \[
 \widehat{N^{j_1}_{\m_k}}<\widehat{N^{j_2}_{\m_k}}<\cdots<\widehat{N^{j_{n_k}}_{\m_k}}<\widehat{N^{j_{n_k+1}}_{\m_k}}
 \]
 in the poset $\modif \widehat{R_{\m_k}}$. By Corollary \ref{ascending stop}, we have $n_k\leq |M^k|-1$. Hence 
 \[\ell\leq 1+\sum_{k=1}^nn_k\leq 1+\sum_{k=1}^n(|M^k|-1).\] 
This proves the first assertion. The second assertion follows from the first one, which implies that any ascending chain in $\modif R$  terminates.
\end{proof}
\end{prop}

The following proves Conjecture \ref{intro conj}\,(1) for terminal singularities of index two. 

\begin{thm}
Let  $\Spec R$ be  a normal affine 3-fold over $\C$ with only terminal singularities of index two. Then $R$ admits an MMA. 
\begin{proof}
Since $R$ is three-dimensional terminal, it has only finitely many isolated singularities $\m_1,\hdots,\m_n$. If the canonical cover $\widehat{S}$ of $\widehat{R_{\m_i}}$ is regular, $\widehat{R_{\m_i}}$ admits an NCCR, which is a Gorenstein MMA. If $\widehat{S}$ is singular, by Corollary \ref{goren mma} and Corollary \ref{MM ex local},  $\widehat{R_{\m_i}}$ admits a Gorenstein MMA. Therefore $R$ admits an MMA by Proposition \ref{local to global}. 
\end{proof}
\end{thm}

\end{document}